\documentclass[11pt]{article}
\usepackage[cp1251]{inputenc}
\usepackage[T2A]{fontenc}
\usepackage[english]{babel}
\usepackage{amsfonts,amsmath,amssymb, amsthm}
\usepackage[pdftex]{graphicx}
\graphicspath{{pictures/}}
\DeclareGraphicsExtensions{.pdf,.png,.jpg}

\newcommand\eps{\varepsilon}
\newcommand\R{{\mathbb R}}
\newcommand\Z{{\mathbb Z}}

\newtheorem{theorem}{Theorem}
\newtheorem{lemma}{Lemma}
\newtheorem{proposition}{Proposition}
\newtheorem{remark}{Remark}

\title{Saddle-center and periodic orbit:\\dynamics near symmetric
heteroclinic connection}
\author{L.M. Lerman$^{a,b}$, K.N. Trifonov$^b$\\
\normalsize $^a$National Research University Higher School of Economics,\\
\normalsize $^{b}$Scientific and Educational Mathematical Center ``Mathematics of Future Technologies’’ \\
\normalsize
Lobachevsky State Research University of Nizhny Novgorod, Russia}
\date{}
\begin{document}
	\maketitle
\begin{abstract}
An analytic reversible Hamiltonian system with two degrees of freedom is studied
in a neighborhood of its symmetric heteroclinic connection made up of a symmetric
saddle-center, a symmetric orientable saddle periodic orbit lying in the same level
of a Hamiltonian and two non-symmetric heteroclinic orbits permuted by
the involution. This is a codimension one structure and therefore it can
be met generally in one-parameter families of reversible Hamiltonian
systems. There exist two possible types of such connections in dependence on how
the involution acts near the equilibrium. We prove a series of theorems which
show a chaotic behavior of the system and those in its unfoldings, in particular,
the existence of countable sets of transverse homoclinic orbits to the saddle periodic orbit
in the critical level, transverse heteroclinic connections involving a pair of saddle
periodic orbits, families of elliptic periodic orbits, homoclinic
tangencies, families of homoclinic orbits to saddle-centers in the unfolding, etc.
As a byproduct, we get a criterion of the existence of homoclinic orbits to a saddle-center.
\end{abstract}

{\bf The orbit structure of a non-integrable Hamiltonian system with two
or more degrees of freedom is very complicated and usually it is
impossible, except for some specific model situations, to describe its structure
more or less completely. By this reason, a fruitful way to understand the orbit
behavior in some parts of the phase space is to detect some simple invariant subsets
(usually containing a finite number of orbits) whose neighborhoods can be understood
from the viewpoint of their orbit structure. When it has been done, we try to find
such structures in general systems and thus to describe partially the
behavior of the system under study. This approach goes back to A. Poincar\'e.
The problem investigated here follow these lines. It was inspired by the
study of stationary waves in a nonlocal Whitham equation \cite{KLM} that
is reduced to the reversible Hamiltonian system with two degrees of
freedom for which homoclinic orbits to different type of equilibria have
to be detected. We rely in this research on earlier results on the
behavior near a homoclinic orbit to a saddle-center equilibrium
\cite{KL,KL1,MHR,GR1,GR,Yag} as well as near
homoclinic tangencies \cite{New,MR,GoSh,Duarte,GTSh,Gon,DGG}.
The results obtained demonstrate how much can be
understood at this approach.

}

\tableofcontents

\newpage

\section{Introduction}

Studying Hamiltonian dynamics is an interesting and hard problem
attracting researchers from many branches of science since Hamiltonian
systems serve as mathematical models in different problems in physics,
chemistry and engineering. The structure of such systems is usually
rather complicated, therefore one of a fruitful approach to these type of
problems is the study of the given system near some invariant sets which can be
selected by simple conditions. Studying systems in neighborhoods of
homoclinic orbits and heteroclinic connections is the problem of such
type. Investigations of dynamical phenomena near a homoclinic orbit to a
saddle periodic orbit was the first such problem, its set up and
understanding the complexity of orbit behavior of the system near such
structure goes back to Poincar\'e \cite{Poin}. The real complexity of
orbit behavior was understood due to works by Birkhoff \cite{Birk}, Smale
\cite{Smale} and finally Shilnikov \cite{SHil}. Other problems, where
complicated dynamics was detected, were studied in many papers by Shilnikov
and coauthors, among them the most influential are
\cite{Shil,Shil1,GSh,Shil3}. Homoclinic dynamics in Hamiltonian systems
began studying in \cite{Dev} where Shilnikov results about the complicated
dynamics near a saddle-focus homoclinic loop were carried over to
Hamiltonian case. The generalization of the Melnikov method onto the
autonomous case for systems close to Hamiltonian integrable \cite{LU1}
allowed one to present examples of a complicated behavior both for
Hamiltonian perturbation of an integrable Hamiltonian system with a
saddle-focus skirt and for dissipative perturbations. The complicated
dynamics near a bunch of homoclinic orbits to a saddle in a Hamiltonian system
was detected in \cite{TurShi} (see generalizations of these results in
\cite{Tur,Hom}). More close to the topic of the present paper results
of \cite{Ler} are where was studied first a complicated dynamics near
a saddle-center homoclinic loop, when the equilibrium was not hyperbolic.
The set-up for this problem was earlier presented in \cite{Conley}, but no essential
results were found there. Results of \cite{Ler} was later extended in
different directions in \cite{KL,KL1,MHR,GR1,GR,Yag}.

In this paper the dynamics is studied in a reversible Hamiltonian system with two
degrees of freedom in a neighborhood of a symmetric
heteroclinic connection which consists of a symmetric saddle-center, a
symmetric saddle periodic orbit which are connected by two nonsymmetric
heteroclinic orbits being permuted by the involution.
(see. Fig.\ref{fig:1}).

\begin{figure}[h]
	\centering
	\includegraphics[width=0.6\linewidth]{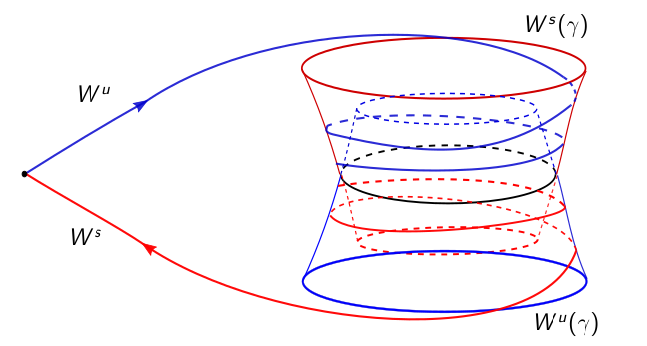}
	\caption{Scheme of the heteroclinic connection}
\label{fig:1}
\end{figure}

This type of connection is a codimension one phenomenon in the class of
reversible Hamiltonian systems. Thus, such a structure can irreparably
appear in one-parameter families. Systems with such structures are met in
applications. For instance, they were discovered in \cite{KLM} where solitons
in a nonlocal Whitham equation \cite{Malkin} were studied.
Also this structure can be found in a one parameter family of reversible Hamiltonian
systems near a destruction of a homoclinic orbit to a saddle-center.
Results of \cite{KL} suggest that saddle periodic orbits accumulate to the
saddle-center loop in the singular level of the Hamiltonian and therefore,
after destruction, unstable separatrix of the symmetric saddle-center can lie on the
stable manifold of a symmetric saddle periodic orbit. Due to reversibility, there is
a pairing stable separatrix of the saddle-center that lie on the unstable
manifold of the same periodic orbit as it is symmetric.

One more application of results obtained is a criterion of the existence
of saddle-center homoclinic loops in a reversible Hamiltonian system (Theorem \ref{loops}
below).

Main results of the paper prove the existence of hyperbolic sets and
elliptic periodic orbits near the heteroclinic connection. Existence of hyperbolic sets
is based on the construction of families of transverse homoclinic orbits
and heteroclinic connections involving two saddle periodic orbits and four
transverse heteroclinic orbits for them. Existence of elliptic periodic
orbits is prove by the same scheme. We look for homoclinic orbits with quadratic
tangencies to saddle periodic orbits in different situations and apply
then results of going back to Newhouse \cite{New,New1} and
Gavrilov-Shilnikov \cite{GSh} and many others \cite{New,MR,GoSh,Duarte,GTSh,Gon,DGG}
on the existence of cascades of elliptic periodic orbits.

\section{Setting up and main notions}

Let $(M,\Omega)$ be a real analytic four-dimensional symplectic manifold, $\Omega$ be its
symplectic 2-form and $H$ be a real analytic function (a Hamiltonian). Such function defines
a Hamiltonian vector field $X_H$ on $M$. Henceforth we assume $X_H$ to have an equilibrium
$p$ of the saddle-center type and without a loss of generality we assume $H(p)=0$.
One more assumption we use is the existence in the same level of $H=0$
a saddle periodic orbit $\gamma$. We use below the notation $V_c=\{x\in M |H(x)=c\}$.

Recall an equilibrium $p$ of $X_H$ on $M$ is called
to be a saddle-center \cite{Ler}, if the linearization operator of the
vector field at $p$ has a pair of pure imaginary eigenvalues $\pm i\omega,\;\omega\in
\R\setminus \{0\},$ and a pair nonzero reals $\pm\lambda\ne 0$. In a
neighborhood of such equilibrium the system has a unique local invariant
smooth two-dimensional invariant symplectic submanifold $W^c_p$ filled with closed orbits
$l_c$ (Lyapunov family of periodic orbits). For the case of analytic $M,
H$ the submanifold $W^c_p$ is real analytic. Near the point $p$ in the level $H=c$
periodic orbit $l_c$ is of saddle type and is located each on its own
level $V_c$.

Also in a neighborhood of $p$ the system has local 3-dimensional center-stable $W^{cs}$ and
center-unstable manifold $W^{cu}$ containing both $p$. These submanifolds contain orbits being
asymptotic, as $t\to\infty$ (for $W^{cs}$), to periodic orbits $l_c$, or as $t\to -\infty$
(for $W^{cu}$), here $W^c = W^{cs}\cap W^{cu}$. Submanifold $W^{cs}$
(respectively, $W^{cu}$) is foliated by levels $V_c$ into local stable (unstable) manifolds
of periodic orbits $l_c$, these submanifolds are diffeomorphic to cylinders $I\times S^1$.
Besides, $W^{cs}$ (respectively, $W^{cu}$), being a solid cylinder, contains as an
axis, an analytic curve $W^s \;(W^u)$ -- local stable (unstable) manifold of the equilibrium $p$,
they consist of $p$ and two semi-orbits tending $p$ as $t\to\infty \;(t\to -\infty)$.

As was supposed above, $V_0$ contains a  saddle periodic orbit $\gamma$. In
the whole $M$ orbit $\gamma$ belongs to a one-parameter family of such
periodic orbits $\gamma_c \subset V_c$ forming an analytic 2-dimensional symplectic
cylinder. Recall that in $V_c$ periodic orbit $\gamma_c$ has two local analytic
2-dimensional Lagrangian submanifolds $W^s(\gamma)$, $W^u(\gamma)$ being
its stable and unstable local submanifolds. All of them are topologically
either cylinders (if its multipliers are positive) or M\"obius strips (if
its multipliers are negative).

Later on in the paper the vector field $v=X_H$ under consideration is supposed to be
reversible as well. This means \cite{Devaney76a} that on $M$ acts a smooth
involution $L: M\to M,$ $L^2 = id_M,$ и $v$ obeys the identity
$DL(v) = - v\circ L.$ For its solutions this property reads as follows:
if $x(t)$ is a solution to $v$, then $x_1(t) = Lx(-t)$ is also its solution.

Orbit $\gamma$ of a reversible $v$ is called {\em symmetric}, if it is invariant w.r.t.
the action of $L$. In particular, an equilibrium $p$, $v(p)=0,$ is symmetric, if
$L(p)=p,$ i.e. $p$ belongs to the fixed point set of $L$, $Fix(L)=\{x\in M |L(x)=x\}$.
The following statement holds true \cite{Devaney76a}.
\begin{proposition}
An orbit of a reversible vector field is symmetric, iff it intersects $Fix(L)$. A symmetric
periodic orbit intersects $Fix(L)$ at two points exactly. The inverse statement is also valid:
if an orbit of a reversible vector field intersects $Fix(L)$ at two different points $x_1, x_2$, then
this orbit is symmetric periodic one and its period is equal to the doubled transition time
from $x_1$ to $x_2$.
\end{proposition}

For a Hamiltonian system the reversibility requires of a clarification, since the involution acts
on the symplectic form
$$
[L^*\Omega](\xi,\eta)= \Omega(DL(\xi),DL(\eta)).
$$
We assume below that an analytic involutive diffeomorphism $L$ is
{\em anti-canonical} mapping, i.e. $L^*(\Omega) = -\Omega$ and $L$ is concordant with
$H$: $H\circ L = H$. In this case the following identities hold:
$$
DL (X_H) = -X_H\circ L,\;\mbox{\rm и}\;	 \Phi_t\circ L = L \circ \Phi_{-t}.
$$
We also assume $Fix(L)$ to be an analytic two-dimensional submanifold in
$M$ (not obligatory connected).

Finally, we assume $X_H$ to have a heteroclinic connection consisting of a
symmetric saddle-center $p$, $H(p)=0$, a symmetric periodic orbit $\gamma$ in the same level of
$H(\gamma)=0$ and two heteroclinic orbits: $\Gamma_1$ going, as $t$
increases, from $\gamma$ to $p$, and $\Gamma_2= L(\Gamma_1)$ going, as $t$ increases,
from $p$ to $\gamma$. Our task is to study the orbit behavior in a neighborhood of this
heteroclinic connection. It is worth observing that the problem, by its set up, is bifurcation,
since the orbit structure varies as a values of the Hamiltonian $c$ varies near a critical value
$c=0$. For instance, on the levels others than $V_0$, the equilibrium is absent
thus the contour is destroyed and bifurcations are expected.

\section{Moser coordinates}

To examine orbit structure of the system near a connection, we shall use convenient
coordinates in a neighborhood of $p$ and in a neighborhood of $\gamma$. Corresponding
results in the analytic case are due to Moser \cite{M,M0}, a finite-smooth version for
a saddle fixed point of a symplectic diffeomorphism exists in \cite{Gon}.
\begin{theorem}\label{t1}
Let $X_H$ be an analytic Hamiltonian vector field and $p$ its equilibrium of the saddle-center
type. Then there is a neighborhood $U$ of $p$, analytic symplectic coordinates $(x_1, y_1, x_2,
y_2)$, $\Omega=dx_1\wedge dy_1+dx_2 \wedge dy_2$, and a real analytic function $h(\xi,\eta)$,
such that $H$ casts in the form
$$
H(x_1,y_1,x_2,y_2) = h(\xi,\eta) = \lambda \xi + \omega \eta + R(\xi, \eta),\qquad  R(\xi, \eta)=
O(\xi^2 + \eta^2),\;
\xi=x_1y_1;\quad \eta =\frac{x^2_2+y^2_2}{2}.
$$
\end{theorem}
To get a symplectic change of coordinates in this theorem, except for results of \cite{M} one needs
to use R\"ussmann's paper \cite{Russmann1}.
\begin{remark}
By a linear scaling of time and, if necessary, a canonical transformation
$y_1\to x_1, x_1\to -y_1$, one can obtain  $\lambda = -1$, and $\omega > 0$ (new $\omega$ is
up to the sign the ratio $\textbar\omega/\lambda\textbar$).
Later on we utilize this normalization.
\end{remark}

Denote $\Phi^t:m\to\Phi^t(m)$ the flow generated by the vector field $X_H$.
In the coordinates of Theorem \ref{t1} the system of differential equations is written down as
\begin{equation}\label{Mo1}
\dot{x_1}=-h_{\xi}x_1,\quad \dot{y_1}=h_{\xi}y_1, \quad\dot{x_2}=-h_{\eta}y_2,
\quad \dot{y_2}=h_{\eta}x_2,
\end{equation}
and its flow $\Phi^t$ is
\begin{equation}
\begin{pmatrix}
x_1(t)\\
y_1(t)\\
x_2(t)\\
y_2(t)
\end{pmatrix}
=
\begin{pmatrix}
\exp[-t\cdot h^0_{\xi}] & 0 & 0 &0\\
0 & \exp[t\cdot h^0_{\xi}] & 0 &0\\
0 & 0 & \cos(t\cdot h^0_{\eta}) & -\sin(t\cdot h^0_{\eta})\\
0  & 0 &\sin(t\cdot h^0_{\eta})& \cos(t\cdot h^0_{\eta})
\end{pmatrix}
\begin{pmatrix}
x_1^0\\
y_1^0\\
x_2^0\\
y_2^0
\end{pmatrix},
\end{equation}
where notations are used
$$h^0_{\xi}=h_{\xi}(\xi_0,\eta_0),\quad h^0_{\eta}=h_{\eta}(\xi_0,\eta_0),\quad \xi_0 =
x^0_1y^0_1,\quad \eta_0=((x^0_2)^2+(y^0_2)^2)/2.
$$

The Hamiltonian system under consideration is reversible as well, hence it is important
to understand to which simplest form can be reduced by means of the same symplectic
coordinate change both the system and the involution in a neighborhood of a saddle-center.
This was done in \cite{GR}. We remind the needed results.
\begin{theorem}
Let $X_H$ be an analytic Hamiltonian vector field and $p$ its equilibrium of the saddle-center
type. Suppose, in addition, $X_H$ be reversible w.r.t. the analytic anti-canonical involution $L$
and $p$ is symmetric. Then in some neighborhood $U$ of $p$ there are analytic
coordinates, as in the Theorem \ref{t1}, such that $L$ has one of two forms:
$$
\begin{pmatrix}
x_1\\
y_1\\
x_2\\
y_2
\end{pmatrix}
\to
\begin{pmatrix}
0 & 1 & 0 & 0\\
1 & 0 & 0 & 0\\
0 & 0 & -1 & 0\\
0 & 0 & 0 & 1
\end{pmatrix}
\begin{pmatrix}
x_1\\
y_1\\
x_2\\
y_2
\end{pmatrix}
$$
or
$$
\begin{pmatrix}
x_1\\
y_1\\
x_2\\
y_2
\end{pmatrix}
\to
\begin{pmatrix}
0 & -1 & 0 & 0\\
-1 & 0 & 0 & 0\\
0 & 0 & -1 & 0\\
0 & 0 & 0 & 1
\end{pmatrix}
\begin{pmatrix}
x_1\\
y_1\\
x_2\\
y_2
\end{pmatrix}
$$
\end{theorem}

\section{Local orbit structure near saddle-center}

Using Moser coordinates is the easiest way to describe the local topology of levels
$V_c$ and the orbit behavior on each level \cite{LU}. The system locally near $p$ takes the form
(\ref{Mo1}). There are two invariant symplectic disks: $x_1=y_1=0$ and $x_2=y_2=0$.
Quadratic functions $\xi = x_1y_1$, $\eta = (x_2^2 + y_2^2)/2$ are local integrals of the system.
Consider the momentum plane $(\xi,\eta)$ in a neighborhood of the origin $(0,0)$. The level $V_c$
of the Hamiltonian corresponds to the analytic curve $\xi = -c +\omega\eta + O(\eta^2+c^2)=
a_c(\eta),$ $0\le\eta\le \eta_*.$ For $c$ small enough these curves form an analytic foliation
of a neighborhood of the origin $(0,0).$ In fact, as $\eta \ge 0,$ one
needs to consider the rectangle $|\xi|\le \xi_0$, $0\le \eta \le \eta_*$ in the momentum plane.

Consider first the level $V_0$, then we get the curve $\xi = \omega\eta +
O(\eta^2)= a(\eta)$ on the momentum plane $(\xi,\eta),$ $0\le\eta\le \eta_*.$
To construct a neighborhood of the origin in $\R^4$ with coordinates
$(x_1,y_1,x_2,y_2)$ we choose four cross-sections $|x_1| = d,$
$|y_1| = d.$ In the manifold $M$ a neighborhood $U$ of the point $p$ in coordinates
$(x_1,y_1,x_2,y_2)$ can be thought as the direct product of two disks $(x_1,y_1,0,0)$
and $(0,0,x_2,y_2)$.

The local structure of $V_0$ is investigated via its foliation into invariant levels of integral
$\eta.$ At $\eta=0$ (the origin on the disk $(0,0,x_2,y_2)$) we have $\xi=a(0)=0$, i.e. we get
a ``cross'' on the disk $(x_1,y_1,0,0)$ (the union of two segments $y_1=0$ and $x_1=0$).
In $U$ the cross coincides with the union of local stable and unstable curves of the
saddle-center $p$. For $\eta > 0$ the value $\xi = a(\eta)$ is positive
and on the disk $(x_1,y_1,0,0)$ we get two pieces of the hyperbola $x_1y_1 = a(\eta),$ which lie
in the first and the third quadrants of the plane $(x_1,y_1,0,0)$, respectively. In $U$
each piece of the hyperbola is multiplied on the circle $x_2^2 + y_2^2 =
2\eta$ on the plane $(0,0,x_2,y_2)$. Varying $\eta$ from zero till $\eta_*$, we get in $U$
two solid cylinders which have a unique common point, the origin, i.e. the saddle-center itself.
Each solid cylinder contains the angle made up of two gluing semi-segments of the cross
($x_1\ge 0$, $y_ 1=0$ and $y_1\ge 0,$ $x_1=0$ for one cylinder and $x_1\le 0$, $y_1=0$
and $y_1\le 0,$ $x_1=0$ for another cylinder). Each such angle is is the topological limit,
as $\eta \to +0$, of cylinders $x_1y_1 = a(\eta),$ $x_2^2 + y_2^2 = 2\eta$, lying in the same
solid cylinder. In particular, each of two solid cylinders contains one half of the stable curve
(a stable separatrix) and one half of the unstable curve (unstable separtatrix) of the
saddle-center. At the fixed $\eta > 0$ each two-dimensional cylinder is
an invariant set and orbits on it go from one of two cross-sections
$|y_1|=d$ to another of two cross-sections $|x_1|=d$ (see. Fig.~\ref{fig:2}).
\begin{remark}
It is worth remarking the property that will be used below. In Moser coordinates
on the level $H=0$ the cross-sections for orbits in the solid cylinder, which is projected onto
the first quadrant of the plane $x_2=y_2=0$, are $y_1=d>0$ (for entering orbits) and
$x_1 = d$ (for outgoing orbits), but for the second solid cylinder, which is projected onto
the third quadrant, they are $y_1=-d$ (for entering orbits) and $x_1 = -d$ (for outgoing
orbits траекторий). This implies that for the case of the first type involution
symmetry permutes the orbit on $y_1>0,$ $x_1=x_2=y_2=0$ with that on $x_1>0,$
$y_1=x_2=y_2=0$ and orbit on $y_1<0,$ $x_1=x_2=y_2=0$ with that on $x_1<0,$
$y_1=x_2=y_2=0$. Hence the symmetry permutes cross-sections from the same solid
cylinder. For the case the second type of involution the symmetry permuts the orbit
on $y_1>0,$ $x_1=x_2=y_2=0$ with that on $x_1<0,$
$y_1=x_2=y_2=0$ and orbit on $y_1<0,$ $x_1=x_2=y_2=0$ with that on $x_1>0,$
$y_1=x_2=y_2=0$. Hence the symmetry permutes cross-sections from the different solid
cylinders.  This will be used below for the classification of heteroclinic
connections.
\end{remark}

A level $V_c$ as $c<0$ corresponds to the curve $\xi = -c + \omega\eta + O(\eta^2+c^2) =
a_c(\eta) > 0$ on the momentum plane $(\xi,\eta)$ for all $0\le\eta\le \eta_*$.
Therefore, the level $V_c$ as $c<0$ consists of two disconnected solid cylinders,
their projections onto the momentum plane form two curvilinear rectangles in the first
and third quadrants bounded by related segments
$|x_1| = d,$ $|y_1| = d,$ and pieces of hyperbolas $x_1y_1 =
a_c(0)=-c >0$ and $x_1y_1 = a_c(\eta_*)$. Each two-dimensional cylinder
$\eta = \eta_0$ is foliated by flow orbits going from one cross-section $|y_1| =
d$ to another one $|x_1| = d$ (see, Fig.~\ref{fig:3}).

For $c>0$ the situation is more complicated, since the curve $\xi = a_c(\eta)$ on the plane
$(\xi,\eta)$ for $0\le\eta\le \eta_*$ correspond to both positive
and negative values of $\xi$ (we consider $|c|$ small enough). Denote $\eta_c$ the unique
positive root of the equation $a_c(\eta)=0$. Then for $0\le\eta\le \eta_c$
pieces of hyperbola $x_1y_1 = a_c(\eta) < 0$ belong to the second and fourth quadrants
of the plane $(x_1,y_1,0,0)$, but for $\eta > \eta_c$ they belong to the first and third quadrants
of this plane. The topological type of the set $V_c$ is a connected sum of two
solid cylinders (i.e. balls). To see this let us consider the projection of $V_c$ on the plane
$(x_1,y_1,0,0)$, it lies inside of the quadrate $|x_1|\le d$, $|y_1|\le d$. let us cut
this set into two halves by the diagonal $y_1 = -x_1.$ Over this diagonal in $V_c$ a 2-dimensional
sphere $S$ is situated. Indeed, extreme points of the diagonal in the second and fourth quadrants
two points of $V_c$ correspond (for them $\eta=0$), but for any point of the diagonal
between extreme points in $V_c$ a circle lies since $\eta>0$ for such points. In particular,
over the point $(0,0)$ of the diagonal the circle $x_2^2+y_2^2 = 2\eta_c$ lies, it is
the Lyapunov periodic orbit $l_c$. Segments $x_1=0$ and, respectively, $y_1=0$ of the plane,
in $V_c$ correspond to stable and unstable manifolds of the periodic orbits.

Subset $V^+_c\subset V_c$ lying over the half-plane $x_1+y_1 > 0$ is composed of three parts.
One part corresponds to the values $\eta_c \le \eta \le \eta_*.$ For the
strict inequality we get a set being diffeomorphic to the direct product
of an open annulus and a segment. As $\eta \to \eta_c+0$ this set has as a topological limit
the set being the direct product of an angle in the plane $(x_1,y_1)$: $0\le x_1\le d$, $y_1=0$
and $0\le y_1\le d,$ $x_1=0$, and  a circle $x_2^2+y_2^2 = 2\eta_c$. Two other parts of
$V^+_c$ are two solid cylinders. These are those subsets in $V^+_c$ that
are projected into the second and third quadrants of the plane
$(x_1,y_1)$, they correspond to $\xi <0$. Every such a solid cylinder is foliated into two-dimensional
cylinders lying over pieces of hyperbola $\xi = a_c(\eta) < 0$, $0\le \eta <
\eta_c,$ $x_1+y_1>0,$ respectively in the second and fourth quadrants. One of the bounding
circles of this cylinder lies over the point on the diagonal (for each cylinder this point is own),
the second bounding circle lies over the point on the segment $y_1=d$ or $x_1=d$.
Solid cylinder projecting onto the second quadrant is glued with its lateral boundary
to the set lying over the first quadrant along the cylinder $x_1=0,\,y_1>0,$ $x_2^2+y_2^2 =
2\eta_c$, but the second solid cylinder which is projecting onto the fourth quadrant, is glued
by its lateral boundary to the set over the first quadrant along the cylinder $y_1=0,\,x_1>0,$
$x_2^2+y_2^2 = 2\eta_c$. Thus, the set is obtained that is homeomorphic to
the solid cylinder from which an inner ball with the boundary $S$ is cut.

Similarly, the second part $V^-_c$ of the set $V_c$ is obtained lying over half-plane $x_1+y_1 \le 0.$
The sets $V^-_c$ and $V^+_c$ are glued along the sphere $S$,
the gluing corresponds to the same points on the diagonal $y_1 = - x_1$.
Visually, this can be imagined in such a way that in each half (solid
cylinder) we cut out by an inner ball and glue the halves obtained along
the boundary of balls in accordance with their orientation. This is a
particular case of a connected sum of two manifolds. Topologically the set obtained
is homeomorphic to a spherical layer  $S^2\times I$. Since each level is
foliated into invariant cylinders $\eta = const,$ we get a complete picture
of the local orbit behavior near a saddle-center (see Fig.~\ref{fig:2}-\ref{fig:4}).

\begin{figure}[h]
	\centering
	\includegraphics[width=0.5\linewidth]{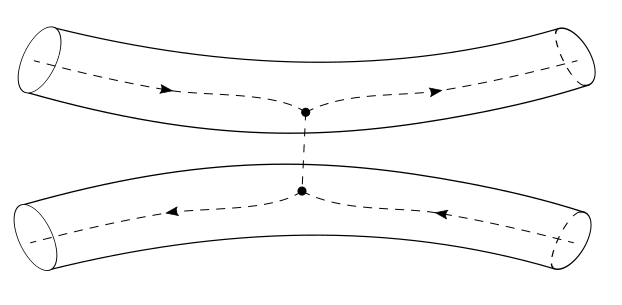}
	\caption{$c=0$}
\label{fig:2}
\end{figure}
\begin{figure}[h]
\centering
\includegraphics[width=0.5\linewidth]{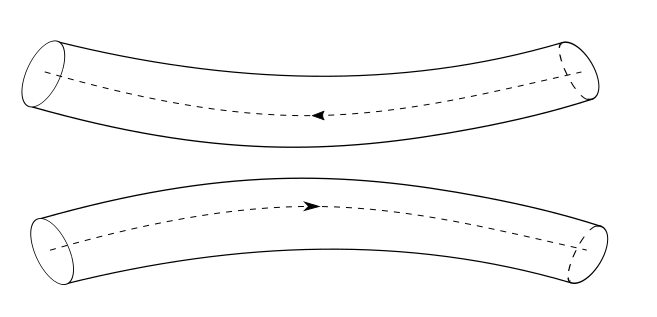}
\caption{$c<0$}
\label{fig:3}
\end{figure}
\begin{figure}[h!]
	\centering
	\includegraphics[width=0.5\linewidth]{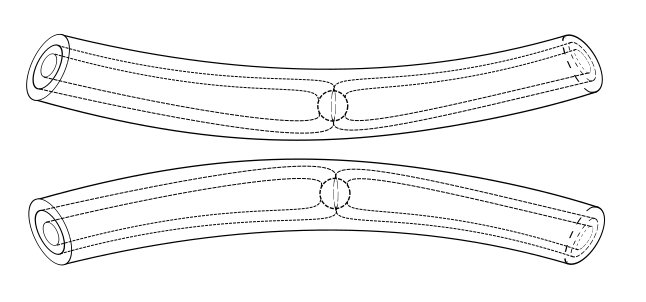}
	\caption{$c>0$}
\label{fig:4}
\end{figure}

\section{Poincar\'e map in a neighborhood of $\gamma$}

To describe the orbit behavior in a neighborhood of a periodic orbit $\gamma$ in the level
$V_0$, we consider a two-dimensional symplectic analytic Poincar\'e map generated by the flow
of $X_H$ on some analytic cross-section $\Sigma$ to $\gamma$. For the case under consideration,
the vector field is reversible, hence the cross-section can be chosen in such a way that
the reversibility would preserve for the Poincar\'e map, as well.

A symmetric periodic orbit intersects submanifold $Fix(L)$ at two points $m_1, m_2$. Take one of them,
$m = m_1,$ and consider a three-dimensional analytic cross-section
$N$ for $\gamma$ containing $m$. $N$ can be chosen in such a way that $m$ belongs to $N$
along with some sufficiently small analytic disk from $Fix(L)$ and $N$ is invariant
w.r.t. the action of $L$. We assume further such choice of the cross-section.

In a sufficiently small neighborhood of $m$ levels $V_c$ form an analytic foliation
into three-dimensional submanifolds since $dH_m \ne 0$. The level $V_0$
contains the curve $\gamma$, but $N$ is transversal to the curve, hence $V_0$ and $N$
intersect each other transversely at the point $m$ and therefore they intersect
in $M$ along an analytic 2-dimensional disk $\Sigma \subset N$. $\Sigma$ is a cross-section
to $\gamma$ in the level $V_0$ and we get an analytic Poincar\'e map $S: \Sigma \to
\Sigma$ with a saddle fixed point $m$.

To study orbit behavior of a system in a neighborhood of $\gamma$ we use
Moser theorem \cite{M0} on the normal form of a 2-dimensional analytic symplectic
diffeomorphism near its saddle fixed point. As $\gamma$ is orientable by the assumption, its
multipliers $\nu, \nu^{-1}$ are positive.
\begin{theorem}[Moser]
In a neighborhood of a saddle fixed point of a real analytic symplectic
diffeomorphism $S$ there are analytic symplectic coordinates $(u,v)$ and an analytic function
$f(\zeta),$ $\zeta = uv,$ $f(0)= \nu,$ such that $S$
takes the following form
\begin{equation}\label{per}
\bar{u}=u/f(\zeta), \bar{v}=vf(\zeta),\;  \mbox{\rm where}\; f(\zeta)= \nu + O(\zeta),\;
0 < \nu < 1.
\end{equation}
\end{theorem}
For our case $S$ is also reversible w.r.t. the restriction of the involution on $\Sigma$,
and involution permutes stable and unstable manifolds (here -- curves) of
the fixed point $m.$ As $X_H(m)\ne 0$, then intersection of $Fix(L)$ and $V_0$ is transverse
at $m$ and it is an analytic curve $l\subset \Sigma$ being the symmetry line
containing $m$. It is not hard to prove, following \cite{Bruno} that symplectic coordinates
$(u,v)$ in the Moser theorem can be chosen in such a way that the restriction of the involution
on $\Sigma$ would act as $(u,v)\to (v,u).$ Then the fixed point set of the involution near point
$m$ coincides with the diagonal $u=v.$ We assume henceforth this to hold.

The orbit $\Gamma_1$ is nonsymmetric and approaches $\gamma$ as $t\to -\infty,$
hence it intersects $\Sigma$ at a countable set of points tending to $m$, but not lying on $l$.
These points belong to the analytic curve $w_u$ being the trace on the disk $\Sigma$
of manifold $W^u(\gamma)$. Similarly, the orbit $\Gamma_2$ intersects $\Sigma$
at countable set of points approaching $m$ at positive iterations of $S$, the points
do not lie on $l$. These points belong to the analytic curve $w_s$ being the trace
of the manifold $W^s(\gamma)$ on $\Sigma$. In Moser coordinates local stable curve
coincides with the axis $v$ (it is given as $u=0$), and local unstable one with the axis $u$ ($v=0$).
Therefore the point $p_s$, the trace of $\Gamma_2$, has coordinates $(0,v_+)$, and
$p_u = L(p_s)$ is the trace of $\Gamma_1$ and has coordinates $(u_-,0)$. To be definite,
we assume $v_+ > 0.$ Then, due to reversibility, one has $u_- = v_+ $.

We choose neighborhoods $\Pi^s,\Pi^u$ of point $p_s, p_u$ defined by inequalities
$\Pi^s: |v-v_+| <\varepsilon,\; |u|<\delta$, и $\Pi^u : |v| < \delta,\;
|u-u_-|<\varepsilon$, the quantities $\delta,\varepsilon$ are small
enough. The set of points from $\Pi^s$ which are transformed to $\Pi^u$ by
some iteration of the map $S$, as is known \cite{Smale,SHil}, consists of
the countable set of strips in $\Pi^s$ accumulating to the stable curve $u=0.$
Due to a convenient normal form, these strips are easily found.
The following assertion holds
\begin{lemma}
Equations $u=f^k(\zeta)(u_- \pm\varepsilon)$ define functions $u=
s^{\pm}_k(v)$, whose domain is $|v-v_+|<\varepsilon$. For them
inequalities $s^{+}_k(v)> s^{-}_k(v)$ hold true
$s^{-}_k(v)> s^{+}_{k+1}(v),$ and $s^{+}_k(v)$ are uniformly tend
to zero as $k\to \infty.$
\end{lemma}
\begin{proof} The proof of this lemma is evident. The lateral boundaries
of strips $\sigma_k^s$ are segments $|v-v_+|=\pm \eps.$ Its upper boundary serves
the curve $s^{+}_k(v)$ providing by solution of the equation $u_k = u_- + \eps$,
and lower boundary is the curve being the solution of the equation $u_k = u_- - \eps$.
To prove the lemma we take an arbitrary $v,$ $|v-v_+|\le \eps$ and find the related
values $u=s^{+}_k(v)$ and $u=s^{-}_k(v)$ from the equations:
$$
u=f^k(\zeta)(u_-+\eps),\;u=f^k(\zeta)(u_--\eps).
$$
Consider, for example, the first equation. Since the value of $\zeta$ preserves
along the orbit of $S$, then multiplying both sides of the first equation
at $v$, we get $g_k(\zeta)=\zeta/f^k(\zeta)= v(u_-+\eps)$. For $k\ge k_0 > 0$
this sequence of functions complex functions has each the inverse one and
all of them are defined in the same disk $|\zeta|\le \sigma$ of the complex plane
$\mathbb C.$ This follows from the complex inverse function theorem, sinc
$g_k(0)=0,$ $g_k'(0)= \nu^{-k}.$ Thus we get $s^+_k(v)= g^{-1}(v(u_-+\eps))/v$ and
$s^+_k(v)\to 0$ as $k\to \infty$ uniformly in $v$.\end{proof}

Functions $s^{\pm}_k(v)$ are upper and lower boundaries of the strip $\sigma_k^s.$
It implies the existence of a countable set of such strips
$$
k>k_0=
E\left\{\frac{ln((\varepsilon+u_-)/\delta)}{ln(\nu^{-1})}\right\}.
$$
Here it is assumed $\varepsilon+u_- > \delta$ (this is the first restriction on the
quantities $\varepsilon,\delta$).

The restriction of $L$ on $\Sigma$ acts as $L:(u,v) \to (v,u)$, hence we get $u_- = v_+ = r$.
Thus we have the same condition on $k_0$ for strips $\sigma^u_k$:
$$k>k_0=E\left\{ \frac{ln((\varepsilon+r)/{\delta}}{ln(\nu^{-1})}\right\}.$$

One more restriction on these quantities provides the requirement that
the neighborhood $\Pi^s$ would not intersect with its image $S(\Pi^s)$, and
$\Pi^u$ with its pre-image $S^{-1}(\Pi^u)$, and $\Pi^s\cap\Pi^u=\emptyset$.
These conditions lead to the inequalities:
$$
\delta< r\frac{1-\nu}{1+\nu},\quad \delta< r-\varepsilon.
$$
Now we can assert, due to construction, that all orbits of the map $S$,
which pass through the points of the neighborhood $\Pi^s$ and reach
the neighborhood $\Pi^u$ for positive iterations, have to pass through
one of strips $\sigma^s_k, k\ge k_0$. these strips has, as its topological limit,
the segment $u=0$ in $\Pi^s$. Its points belong to the stable manifold and
they tend under $S^n$ to the fixed point $m$ as $n\to \infty.$

From the reversibility of $S$ and the same considerations we get that
images of strips $\sigma^s_k$ in $\Pi^u$, i.e. strips $\sigma^u_k$,
accumulate as $k\to \infty$ to the points of the segment $v=0$ in $\Pi^u$.
These points under negative iterations of $S$ tend to $m$.

\begin{remark}
It is worth remarking the useful fact. At a given small value of $c$ the
Hamiltonian system in a neighborhood of $\gamma$ has an analytic invariant 3-dimensional
submanifold $V_c = \{H=c\}$. In $V_c$ a saddle periodic orbit $\gamma_c$ lies being a continuation
in $c$ of the orbit $\gamma$. The family $\gamma_c$ makes up an analytic
2-dimensional local symplectic cylinder containing $\gamma_0 = \gamma$.
Chosen above 3-dimensional cross-section $N$ under its intersection with $V_c$ gives
an analytic 2-dimensional symplectic disk $\Sigma_c$ being the local cross-section
for the restriction of the flow on $V_c$. The Poincar\'e map $S_c$ on $\Sigma_c$
is symplectic analytic with the saddle fixed point $m_c$. For this map
Moser theorem as also valid and its can be transformed to the form
(\ref{per}). Moreover, since the dependence on $c$ is analytic the change of
variables can be done for all $c$ small enough at once and in (\ref{per})
function $f$ will depend on $c$ analytically. This will be used below to study
the dynamics on $V_c$ for $c$ close to $c =0$.
\end{remark}

\section{Global maps}

Now we derive the representations of the global maps $T_1$ generated by the flow near
$\Gamma_1$ acting as $T_1:\Pi^u \to D_1$. It is analytic symplectic diffeomorphism
and is written as $x_2 = f(u,v),\quad y_2 = g(u,v)$, here symplecticity is equivalent to
the identity $f_u g_v - f_vg_u \equiv 1$ (area preservation).

Linearization of this map at the point $(u_-,0)$ has the form $x_2=\alpha (u-u_-)+\beta
v,$ $y_2=\gamma (u-u_-)+\delta v$, where $\alpha=f_u,\quad  \beta=f_v,\quad \gamma=g_u, \quad
\delta=g_v$, all derivatives are calculated at the point $(u_-,0)$. Thus, a
general form of $T_1$ is
$$ x_2=\alpha (u-u_-)+\beta v +\cdots, \quad  y_2=\gamma (u-u_-)+\delta v + \cdots.$$

The system under study is reversible and cross-sections are chosen consistently with the action
of involution, then the global map $T_2: D_2 \to \Pi^s$ near
$\Gamma_2 = L(\Gamma_1)$ is expressed as $T_2=L\circ T_1^{-1}\circ L$, or in coordinates:
\begin{eqnarray}
u_1=\gamma\bar{x}_2 + \alpha\bar{y}_2 + \cdots, \nonumber \\
v_1-v_+ = -\delta\bar{x}_2 - \beta\bar{y}_2 + \cdots.
\nonumber
\end{eqnarray}

Below, when studying of the orbit behavior on the levels $V_c$ for $c\ne 0$, we shall need
to know the form of the global maps in these cases. As was mentioned
above, without loss of generality, we can regard as coordinates on the disks
$D_1(c), D_2(c)$ symplectic coordinates $(x_2,y_2)$ and $(\bar{x}_2,\bar{y}_2)$,
respectively, and on the disk $\Sigma(c)$ symplectic coordinates $(u,v)$.
Global maps are analytic symplectic diffeomorphisms analytically depending
on $c.$ Thus they have the form
\begin{equation}\label{gloc}
\begin{array}{l}
T_1(c):\;x_2=a(c)+\alpha(c) (u-u_-)+\beta(c) v +\cdots, \;y_2= b(c)+\gamma(c) (u-u_-)+
\delta(c)v +
 \cdots \\\\
T_2(c):\;u_1=a_1(c)+\gamma(c)\bar{x}_2 + \alpha(c)\bar{y}_2 + \cdots, \;
v_1-v_+ = b_1(c) -\delta(c)\bar{x}_2 - \beta(c)\bar{y}_2 + \cdots,
\end{array}
\end{equation}
where $a_1(c) = \gamma(c)a(c)-\alpha(c)b(c),\; b_1(c)= -\delta(c)a(c)+\beta(c)b(c).$

\section{Types of symmetric connections}

The symmetric periodic orbit $\gamma$ in the level $V_0$ is outside of a neighborhood
of point $p$, so one needs to conform the location of this orbit and its symmetry
with the action of involution in $U$ relative to coordinates. This concordance is
relied on the existence of connecting orbits $\Gamma_1$ and $\Gamma_2 = L(\Gamma_1)$.

Near the point $p$ involution $L$ permutes local stable and unstable curves of $p$.
Since heteroclinic orbits contain these curves, two cases are possible here.
To understand this we remind that we have chosen in $U$ on the level $V_0$
two smooth disks $D_1,$ $D_2 = L(D_1)$ being transverse to orbits $\Gamma_1$ and $\Gamma_2$,
respectively. In Moser coordinates near $p$ we can take as such disks cross-sections
$y_1=\pm d$ and $x_1=\pm d$, where the sign is determined by the intersection of
a respective cross-section with $\Gamma_1$ and $\Gamma_2$. On $V_0$
coordinates on the disks are $(x_2,y_2),$ since a coordinate conjugated to
$y_1$ (or, respectively, to $x_1$) is found from the equality
$H=0$. Recall  (see above) that in a neighborhood of point $p$ the local topological type of
the level $V_0$ is a pair of 3-dimensional solid cylinders with two their inner
points identified chosen by one in each cylinder (after gluing this is the
point $p$) (see Fig.~\ref{fig:2}). The lateral boundary of each solid cylinder
is a smooth 2-dimensional invariant cylinder, two other boundaries are two disks
(``lids''). For each solid cylinder orbits enter through one lid and leave
the cylinder through the other lid.

One can regard the cross-sections $D_1,$ $D_2$ be two lids of these cylinders
and what is more, $D_1$ is the entry disk and $D_2$ is exit one. Two different
situations are possible: 1) both orbits $\Gamma_1$, $\Gamma_2$ belong locally to the same
solid cylinder, this is equivalent to the conditions that both disks $D_1,$ $D_2$
are on the boundary of the same cylinder; 2) orbit $\Gamma_1$ belongs locally to
one solid cylinder but $\Gamma_2$ belongs locally to another solid cylinder, that is,
disks $D_1$ and $D_2$ belong to the boundaries of different cylinders.
In Moser coordinates the case 1 corresponds to the action $L:(x_1,y_1,x_2,y_2)\to$
$(y_1,x_1,-x_2,y_2)$, but the case 2 does to the action $L:(x_1,y_1,x_2,y_2)\to$
$(-y_1,-x_1,-x_2,y_2)$. The case 1 means the invariance of the related cylinder w.r.t.
the involution and the case 2 means its permutability (one cylinder transforms to another one).
In case if $L$ preserves the cylinder, the intersection of $Fix(L)$ with the cylinder is a
curve, but if $L$ permutes cylinders this intersection is the only point $p$
(see, Fig.~\ref{fig:5}-\ref{fig:6}).

Now consider those orbits of the vector field which enter through $D_1$ near $\Gamma_1$,
but differ from $\Gamma_1$. As $t$ increases, they enter into the related solid cylinder,
pass it and leave the cylinder (the semi-orbit $\Gamma_1$ itself tends to $p$ and stays in
the cylinder). These orbits either intersect $D_2$ (case 1), or leave the cylinder without
intersecting $D_2$ (case 2). In the case 2 the Poincar\'e map is not defined in a neighborhood of
the connection for $c \le 0$, since orbits close to $\Gamma_1$ do not return on $D_2$,
if a system under consideration does not fulfil some additional global conditions (of the type
the existence of a homoclinic orbit joining two remaining lids).
\begin{figure}[h]\label{cases}
	\centering
	\includegraphics[width=0.6\linewidth]{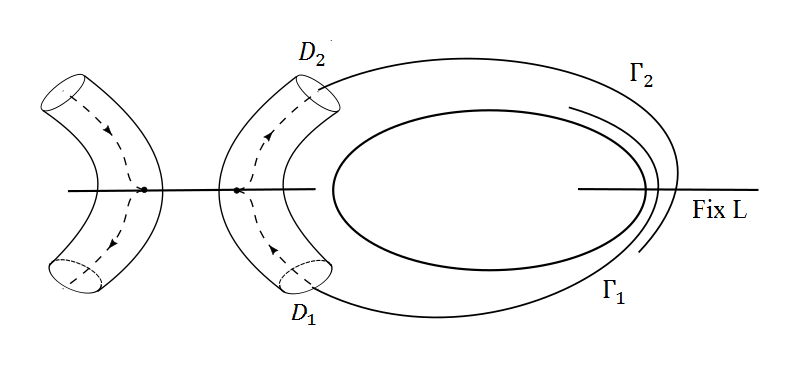}
	\caption{Case 1}
\label{fig:5}
\end{figure}
\begin{figure}[h]
	\centering
	\includegraphics[width=0.5\linewidth]{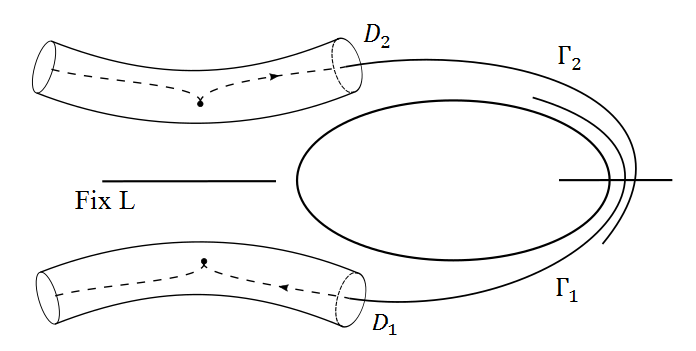}
	\caption{Case 2}
\label{fig:6}
\end{figure}

\begin{remark}
For $c > 0$ small enough for the case 2 the Poincar\'e map on the related disk $D_1(c)$
becomes defined inside of some small disk centered at $(0,0)$ whose boundary circle is the trace of
a stable manifold $W^s(l_c)$ of the Lyapunov periodic orbit $l_c$. Then
for some sequence of positive $c_n \to 0$ the trace on the disk $D_1(c_n)$ of
$W^s(\gamma_{c_n})$ is tangent to the trace of $W^u(\gamma_{c_n})$, what is accompanied
with the appearance of elliptic periodic orbits. This will be considered below.
\end{remark}

\section{In the singular level $V_0$}

We start with the case 1 at $c=0.$ In the neighborhood $U$ of $p$ we have in Moser coordinates
the representation $h=-\xi+\omega\eta+ R(\xi,\eta)$, hence the manifold $W^{cs}$ is given as
$x_1=0$, $W^{cu}$ as $y_1=0$, $W^s$ by the equalities $x_1=x_2=y_2=0$, and
$W^u$ by $y_1=x_2=y_2=0$. Suppose, to be definite, that in $U$ heteroclinic orbit $\Gamma_1$
approaches to $p$ for values $y_1>0$, i.e. in the level $V_0$ disk $D_1$ is defined by the equality
$y_1=d$ and disk $D_2$ by the equality $x_1 =d.$ In $U$ signs of variables
$x_1$ and $y_1$ preserve by the flow and 3-dimensional cross-section
$N^s: y_1 = d > 0, |x_1|\leq\delta, \eta\leq\eta_0,$ is transverse to $\Gamma_1$ and to all orbits
close to $\Gamma_1$, due to the inequality $h_{\xi}=-1+\dots \ne 0$ in $U$.
Similar assertions are valid for the cross-section $N^u = L(N^s): x_1=d, |y_1|\leq\delta,
\eta\leq\eta_0$. Each cross-section is foliated by levels $H=c$ into disks, one of which is
$D_1 = V_0\cap N^s$ and, respectively, $D_2 = V_0\cap N^u$.

Denote as $a(\eta)$ the solution of the equation $h(\xi,\eta)=0$ w.r.t. $\xi,\; \xi = a(\eta)=
\omega\eta + O(\eta^2)$. One may regard that when variables $(x,y)$ vary in $U\cap V_0$,
corresponding solutions of the equation $h(\xi,\eta)=0$ lie on the graph of function $a$.
Then 2-disk $D_1$ in $N^s$ is the graph of the function $x_1=a(\eta)/d$, and 2-disk $D_2$
in $N^u$ is the graph of function $y_1=a(\eta)/d$.
Both $D_1$, $D_2$ are analytic disks being symplectic w.r.t. the restriction of 2-form
$\Omega$ on $D_1$ and $D_2$, respectively, and the local map $T: D_1\to D_2$ generated
by the flow $\Phi^t$ is symplectic.
\begin{remark}
For the type 2 of the involution the cross-sections are $y_1=d$ (for $N^s$) and $x_2=-d$
(for $N^u$). Therefore, orbits from $D_1= N^s\cap V_0$ hit $D_2= N^u\cap
V_0$ only if $a(\eta)<0.$
\end{remark}
Let us find an explicit representation of the map $T$ in coordinates $(x_2,y_2)$.
The passage time $\tau$ for orbits from $N^s$ to $N^u$ is found from (\ref{Mo1}), where
$x_1(\tau)=d,$ $y_1(0)=d$: $\tau = -(h_\xi)^{-1}\ln(d/x_1), x_1 = a(\eta)/d$.
From (\ref{Mo1}) it follows that $T$ has the form
\begin{equation}\label{T}
\bar{x}_2 = x_2\cos\Delta(\eta)-y_2\sin\Delta(\eta), \quad \bar{y }_2 =
x_2\sin\Delta(\eta)+y_2\cos\Delta(\eta),
\end{equation}
with
\begin{equation}\label{Del}
\Delta(\eta)= -\frac{h_\eta}{h_\xi}\ln(d/x_1)= a'(\eta)\ln(d^2/a(\eta))=
(\omega + O(\eta))\ln(d^2/a(\eta)).
\end{equation}
\begin{remark}\label{rm:case2}
For the type 2 of the involution the formula is modified as $\Delta(\eta)=
a'(\eta)\ln(-d^2/a(\eta))$.
\end{remark}
Our first result is the following theorem
\begin{theorem}
If an analytic reversible Hamiltonian system has a heteroclinic connection  of the type 1
with properties indicated, then the saddle periodic orbit $\gamma$ has a countable set
of 1-round transverse homoclinic orbits. In the case of the type 2 connection no other
orbits exist in $V_0$ in a sufficiently small neighborhood of the connection, except
for orbits of the connection.
\end{theorem}
To be precise, let us make more exact the notion of a 1-round homoclinic
orbit for $\gamma$. To this end, consider in $V_0$ a sufficiently small tubular neighborhood
of the orbit $\gamma$. Since $V_0$ is orientable, this neighborhood is homeomorphic to a
solid torus $D^2\times S^1$. The union of point of the orbit $\Gamma_1$, point $p$ and points
of the orbit $\Gamma_2$ give a simple non-closed curve without self-intersections in $V_0$.
One may regard that this infinite curve consists of three connected pieces, one of which, $R$,
lies outside of the tubular neighborhood of $\gamma$, and two remaining ones are inside of
this tubular neighborhood (recall that orbits $\Gamma_1$, $\Gamma_2$ tend asymptotically to
$\gamma$). Now consider a homoclinic orbit to $\gamma$, whose global part outside of the tubular
neighborhood of $\gamma$ belongs to a small neighborhood of the curve $R$, but two remaining
parts are inside of the tubular neighborhood. Such homoclinic orbit for $\gamma$ will be called
1-round one.

\begin{proof} To prove the theorem we will show that a segment of the unstable separatrix
$w_u\cap \Pi^u$ of the saddle fixed point $m$ on $\Sigma$ is transformed by the the map
$T_2\circ T\circ T_1$ into an analytic curve that intersects transversely at the countable
set of points the segment $w_s\cap \Pi^s$ of the unstable separatrix $w_s$ of the same fixed point.

Consider in $\Pi^u: v=0,\;|u-u_-|\le \varepsilon_1< \varepsilon,$ a segment $(A,B)$ of
the curve $w_u$. Its image under the action of $T_1$ is a parameterized
curve on the disk $D_1$: $x_2(\tau) =\alpha\tau + \cdots,$ $y_2(\tau)=\gamma\tau + \cdots$,
its parameter is $\tau = u-u_-$. Since $T_1$ is a diffeomorphism, we get a smooth curve in
$D_1$ passing through $(0,0)$, its tangent vector at $(0,0)$ is nonzero vector
$(\alpha,\gamma)$. Boundary points of this curve denote as $A_1,B_1$ and the curve obtained
as $[A_1,B_1]$.

The curve $[A_1,B_1]$ by the map $T$ is transformed to the spiral-shape curve on the disk $D_2$:
$$
\begin{array}{l}
\bar{x}_2=x_2(\tau)\cos\Delta(\eta(\tau))-y_2(\tau)\sin\Delta(\eta(\tau)), \\
\bar{y}_2=x_2(\tau)\sin\Delta(\eta(\tau))+y_2(\tau)\cos\Delta(\eta(\tau)).
\end{array}
$$
In symplectic polar coordinates on $D_1, D_2$, respectively,
$$
x_2=\sqrt{2\eta}\cos\phi,\quad y_2=\sqrt{2\eta}\sin\phi,\;
\bar{x}_2= \sqrt{2\bar\eta}\cos\theta,\quad \bar{y}_2=\sqrt{2\bar\eta}\sin\theta,
$$
the map $T$ has the form
$$
\bar{\eta}=\eta,\quad \theta = \phi+\Delta(\eta)\;(\mbox{\rm mod}\,2\pi).
$$
This map is defined for values $\eta > 0.$ Under the action of $T$ the curve $[A_1,B_1]$
transforms into two infinite spirals corresponding to $\tau > 0$ and $\tau < 0$
$$
\overline{\eta}=\eta(\tau),\quad \theta=\phi(\tau) + \Delta(\eta(\tau)),
$$
where for $|\tau|$ small enough we have for $\alpha \ne 0$
$$
\eta(\tau) = (x_2^2(\tau)+y_2^2(\tau))/2 = \frac{\alpha^2+\gamma^2}{2}\tau^2 +
O(\tau^3),\;\tan\phi(\tau) = \frac{\gamma}{\alpha} + O(\tau),
$$
and for $\alpha = 0$ the angle is defined via $\cot\phi$, here the values $\phi$
as $\tau \to +0$ and $\tau \to -0$ differ by $\pi.$ Since $\phi$ is bounded as $\tau \to \pm 0$,
but function $\Delta(\eta(\tau))$ monotonically increase to $\infty,$ then each of
spirals, as $|\tau| \to 0,$ tends to $(0,0)$ on $D_2$, making infinite number of rotations
in angle: $\theta(\tau)\to \infty.$ Take a segment on $u=0$ symmetric to $[A,B]$ and its
$T_2$-pre-image $[A_2,B_2]$ on $D_2$, it is an analytic segment through the point $(0,0)$
symmetric to $[A_1,B_1]$. Therefore, it intersects each spiral at the countable set of
points through which orbits pass tending to $\gamma$ as $t\to \pm \infty$, that is,
they are Poincar\'e homoclinic orbits \cite{SHil}. To complete the proof, we need to show
transversality of intersections spirals and $[A_2,B_2]$. Instead, we shall prove
the transversality of $T_2$-images of spirals and the segment $u=0$ on $\Pi^s$.

Consider, for instance, one of spirals, defined by inequality $\tau > 0$ and apply $T_2$
\begin{eqnarray}\label{glob}
 u_1=\gamma\bar{x}_2 + \alpha\bar{y}_2 + \cdots =
\sqrt{2\eta(\tau)}\sqrt{\alpha^2+\gamma^2}\left[\sin(\varphi(\tau)+\Delta(\eta(\tau))+\sigma)+
O(\sqrt{2\eta(\tau)})\right],
\nonumber \\
v_1-v_+ = -\delta\bar{x}_2 - \beta\bar{y}_2 + \cdots = \sqrt{2\eta(\tau)}
\sqrt{\beta^2+\delta^2}\left[\sin(\varphi(\tau)+\Delta(\eta(\tau))+\sigma_1+O(\sqrt{2\eta(\tau)})\right].
\nonumber
\end{eqnarray}
The map $T_2$ transforms the spiral and the point $(0,0)$ from $D_2$ to some spiral-shape curve and
the point $(0,v_+)$ in $\Pi^s$. We need to prove that the spiral obtained does not tangent
to the segment $u=0$ at any common point. To this end, we show that the derivative $u_1'(\tau)$
does not vanish at the intersection points of the spiral with the segment
$u=0$ in $\Pi^s$. As $\eta(\tau)\ne 0,$ zeros of the function $u_1(\tau)$ are determined
by zeros of the function $\sin(\varphi(\tau)+\Delta(\eta(\tau))+\sigma)$, and one needs to check
the inequality $u_1'(\tau)\ne 0$ for those $\tau$ where $u_1=0$.

The derivative $u_1'(\tau)$ at points where $\sin(\varphi(\tau)+\Delta(\eta(\tau))+\sigma)=0$,
is equal up to a nonzero multiplier
$$
\cos(\varphi(\tau)+\Delta(\eta(\tau))+\sigma)(\varphi'(\tau)+
\Delta'(\eta(\tau))\eta'(\tau)).
$$
Thus, the first multiplier is nonzero and the principal term in the
bracket for small enough $\tau$ is $\Delta'(\eta(\tau))\eta'(\tau),$
that tends to infinity as $\tau \to 0$. Indeed, in accordance to formula (\ref{Del})
for $\Delta$ we have
$$
\displaystyle{\Delta'(\eta)= a''(\eta)\ln(d^2/a(\eta))- \frac{a'^2(\eta)}{a(\eta)}=
\frac{-a'^2(\eta)+a''(\eta)a(\eta)\ln(d^2/a(\eta))}{a(\eta)}},
$$
hence the numerator is negative and separated from zero for small $\eta$, but the denominator
tends to zero as $\eta \to +0.$ The ratio $\eta'(\tau)/a(\eta(\tau))$ is
of the order $1/\tau$. Therefore, the existence of a countable set of transverse homoclinic orbits
has been proved.

For the case 2 orbits of the system passing on $\Pi^u$ through the points of unstable curve
$v=0$, $|u-u_-| < \varepsilon$, as $t$ increases, intersect disk $D_1$ and after that leave
$V_0$ (see Fig.~\ref{fig:6}). The same holds true, due to symmetry, as $t$
decreases, for orbits passing on $\Pi^s$ through the points of stable curve
$u=0, |v-v_+| < \varepsilon$. \end{proof}

\begin{figure}[h]
	\centering
	\includegraphics[width=0.6\linewidth]{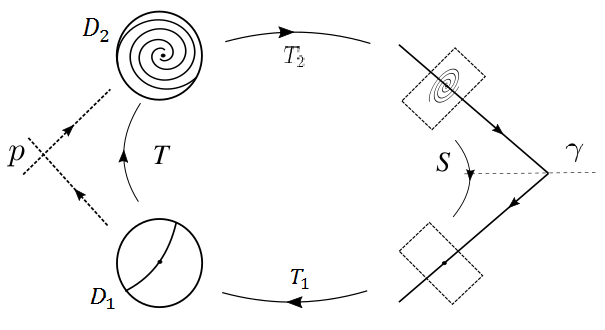}
	\caption{Poincar\'e map as $c=0$.}
\label{fig:7}
\end{figure}
The proven theorem allows one to use results \cite{Smale,SHil} about the
orbit structure near a transverse homoclinic orbit of a two-dimensional
diffeomorphism. Namely, near each homoclinic orbit there exists its
neighborhood such that orbits of a diffeomorphism passing through this
neighborhood make up an invariant hyperbolic subset whose dynamics is
conjugated with the shift on a transitive Markov chain (see, for instance,
\cite{KH}). On $\Pi^s$ we have a countable set of different homoclinic orbits
accumulating at the trace of heteroclinic orbit $\Gamma_2$. It is clear,
for a fixed a homoclinic point from the set the size of a neighborhood,
where the description holds, tends to zero as homoclinic points approach
to the trace of $\Gamma_2$. If we consider an only finite number of homoclinic
points outside of a small neighborhood
of the trace of $\Gamma_2$, then we get a uniformly hyperbolic set generated by
these homoclinic orbits. Therefore, the entire region where hyperbolic set for our case exists
should be of a two-horn shape bounded by two parabola-like curves which are tangent
at the point $(0,v_+).$ The strips near homoclinic orbits (see above) for different
homoclinic points interact each other under iterations of the Poincar\'e map
that lead to a Markov chain with a countable set of states, but the invariant set
obtained in this way is not uniformly hyperbolic but only non-uniformly hyperbolic.
\begin{figure}[h]
	\centering
	\includegraphics[width=0.6\linewidth]{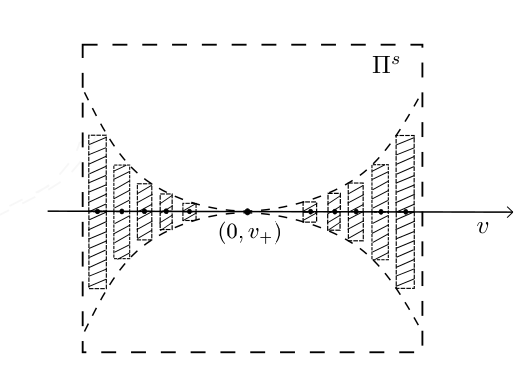}
	\caption{The shape of non-uniform hyperbolicity region in $\Pi^s.$}
\label{fig:8}
\end{figure}

\section{Hyperbolicity and ellipticity in levels $c<0$}

In this section we consider levels $V_c, c < 0,$ near the connection for the case 1.
For the case 2 and $c<0$ all orbits, entering through $D_1(c)$ to a neighborhood, leave it,
the same is true for the orbits entering a neighborhood, as $t$ decreasing, through $D_2(c)$.

We prove: 1) existence of a hyperbolic set constructed on a finite number of
transverse homoclinic orbits to the saddle periodic orbit $\gamma_{c}$; 2) existence
of a countable set of intervals of values $c<0$ accumulating at zero whose values of $c$
correspond to levels where in $V_c$ a 1-round elliptic periodic orbit exists. Remind that
for the case 1 and negative $c$ small enough all orbits in a neighborhood
of $p$ passing through one ``lid'' $D_1(c) = N^s\cap V_c$ of the solid cylinder,
as $t$ increases, intersect its second lid $D_2(c) = N^u\cap V_c$.

Existence of finitely many transverse homoclinic orbits to a periodic orbit $\gamma_{c}$
is almost evident and follows from their existence at $c=0.$ For any negative $c$ small enough
we consider that solid cylinder of the local part of the level $V_c$ near $p$, whose
lids are $D_1(c), D_2(c).$ Then global maps transform: a segment $w_u(c)$ on
$\Pi^u(c)$ (i.e. $v=0$) is mapped by $T_1(c)$ onto a curvilinear segment on $D_1(c)$ passing near point
$(0,0)$ at the distance of the order $|c|^l,$ $l\ge 1$. The same holds true for a
symmetric curve on $D_2(c)$ being the pre-image w.r.t. $T_2(c)$ of the
segment of $w_s(c)$ (i.e. $u=0$). Let us cut out on $D_1(c),$ $D_2(c)$
small disks of the radius of an order $O(\sqrt{|c|})$ centered at points $(0,0)$.
Since the map $T(c)$ preserves $\eta$, we cut out thereby by one interval on each
curvilinear segment of each disk. After cutting out two remaining segments stay on each disk.
Consider the images of the remaining segments on $D_2(c)$ under the map $T_2(c)\circ
T(c)$.
\begin{theorem}
For $|c|$ small enough the image of each remaining segment is a finite spiral
which intersects transversely at a finite number of points the curve $u=0$ in $\Pi^s(c)$.
\end{theorem}
\begin{proof} We follow the lines of the case $c=0$. The difference is
that we first cut out the disk on $D_1(c)$ of the radius $O(\sqrt{|c|}),$ $\eta\le \eta(c)$,
with the center at $(0,0)$. The image of the segment $v=0$ on $\Pi^u(c)$
under the action of the map $T_1(c)$ is a smooth curve passing at the
distance of the order $c^l,$ $l\ge 1$, from the point $(0,0)$ (the case when this curve
passes through the point $(0,0)$ is not excluded). Therefore the circle of
the radius $\eta = \eta(c)\sim |c|$ intersects this curve at two points, i.e.
parts of this curve lying outside of the circle are two smooth segments
and their images w.r.t. $T(c)$ are two finite spirals which intersect
transversely the $T_2(c)$-pre-image of the segment $u=0$ from $\Pi^s(c)$. Thus, we get
a finite number of transverse homoclinic orbits for $\gamma_c$. Obviously, the less
$|c|$, the more number of transverse homoclinic orbits can be found.
\end{proof}

To prove the existence of elliptic points in some neighborhood of the connection in
the whole $M$, we first find a countable set
of values $c_n < 0$, for which the system in the level $V_{c_n}$ has a
non-transverse homoclinic orbit with quadratic tangency for periodic orbit $\gamma_{c_n}$.
This allows one to apply results on the existence of cascades of elliptic periodic orbits
on the levels close to $V_{c_n}$ (see, for instance, \cite{Gon,DGG}).
\begin{theorem}
Suppose inequality $a_1'(0)\ne 0$ to hold. There is a sequence $c_n \to -0$
such that in the level $V_{c_n}$ periodic orbit $\gamma_{c_n}$ possesses
a homoclinic orbit along which stable and unstable manifolds
$W^s(\gamma_{c_n})$, $W^u(\gamma_{c_n})$ have a quadratic tangency. For every such $c_n$
there exists a countable set of $c$-intervals $I_{nm}$, $I_{nm} \to c_n$ as $m\to \infty,$
whose values $c\in I_{nm}$ represent levels where the system has a 1-round elliptic
periodic orbit in this level of the Hamiltonian.
\end{theorem}
\begin{proof} The inequality $a_1'(0)\ne 0$ is of the general position
condition. It is the analog of the condition C in \cite{Gon}. This guarantees
that for $|c|$ small enough the $T_2(c)$-image of the
point $(0,0)$ in $\Pi^s(c)$ is an analytic curve intersecting transversely
stable manifold $u=0$ of the saddle fixed point. To be precise, we remind
we assume coordinates $(u,v)$ not depending on $c$ only maps do.
As a corollary of this inequality, by reversibility, there is a $c_0 < 0$ such that
for $c\in (c_0,0)$ the $T_1(c)$-image of the segment $v=0$ is a smooth curve
that does not pass through the point $(0,0)$ on $D_1(c)$ and the distance from $(0,0)$
to this curve is of the order $|c|.$

To prove the first assertion of the theorem, we consider
level $V_c$ for small negative $c$ and find the image of the segment $v=0$
from $\Pi^u(c)$ under the map $ T(c)\circ T_1(c)$. This is an analytic curve
in $D_2(c)$. One needs to show that this curve for a countable set of $c$-values touches
the $T_2(c)$-pre-image of the segment $u=0$ from $\Pi^s(c)$.

Let us write down the representation of the map $T(c)$. It is similar to (\ref{T})
but for $c<0$ function $\Delta_c(\eta)$ is analytic and has the form
$$
\Delta_c(\eta)= a'_c(\eta)\ln{\frac{d^2}{a_c(\eta)}},\;a_c(\eta)= -c +
\omega\eta+O_2(c,\eta) > 0.
$$
The positivity of the function $a_c(\eta)$ implies the map $T(c)$ be a
local analytic symplectic diffeomorphism in some neighborhood of the
point $(x_2,y_2)=(0,0)$ for all sufficiently small in modulus negative $c$.

The map $T_1(c)$ is also analytic, hence the $T_1(c)$-image of the segment $v=0,
|u-u_-|\le \varepsilon$  be an analytic curvilinear segment in $D_1(c)$ passing
near point $(x_2,y_2)= (0,0)$ at the distance of the order $|c|$. This follows
from the genericity assumption $a_1'(0)\ne 0$ and symmetry of $T_1(c)$ and $T_2(c)$.
By symmetry, the $T_2(c)$-pre-image of
the segment $u=0, |v-v_+|\le \varepsilon$ is also an analytic curvilinear
segment in $D_2(c)$ being symmetric w.r.t. $L$ to the segment in $D_1(c)$ and passing
near the point $(\bar{x}_2,\bar{y}_2)= (0,0)$ at the same distance of the order $|c|$.

In polar coordinates on disks $D_1(c)$, $D_2(c)$ the map $T_c$ has of the form
$$
\overline{\eta}=\eta,\quad \overline{\theta}=
\varphi+\Delta_c(\eta),
$$
with $\Delta_c(\eta)=(\omega+\cdots)\ln[d^2/(-c+\omega\eta+\cdots)].$

Expanding in formulas for $T_1(c)$ coefficients by the Taylor formula up to the terms of the
first order in $c$ we  get $a(c)=ac+...,\;a\ne 0,\; b(c)=bc+...$. Then one has
$$
\Delta_c(\eta)= -\omega \ln \frac{d^2}{-c+\omega[(ac+\alpha(u-u_-))^2+(bc+
\gamma(u-u_-))^2]/2}+O_2(c,\eta).
$$
On the disk $D_1(c)$ the curvilinear segment under consideration is an analytic smooth curve
at the distance of the order $|c|$ from $(0,0)$, so there is a circle $\eta = \eta_c$ such
that this circle and the curve have a common point and they are tangent at this
point. In principle, this point can be not unique. Other points of this curve
are outside of this circle.

Local map $T(c)$ preserves $\eta$, hence the $T(c)$-image on $D_2(c)$ of the curve
 is a spiral-shape curve that lies outside of the circle
$\eta=\eta_c$ on $D_2(c).$ By symmetry, on the same circle on $D_2(c)$
there are other its points of tangency with the curve being $T_2(c)$-pre-image
of the segment $u=0$ from $\Pi^s(c).$ An important observation is the
following assertion.
\begin{lemma}\label{tn}
For $c$ small enough the only point of tangency the circle and the curve
on $D_1(c)$ exists. The tangency at this point is quadratic.
\end{lemma}
\begin{proof}
Denote $\sigma_c^s$, $\sigma_c^u$ circles $\eta=\eta_c$ on $D_1(c)$, $D_2(c)$, respectively.
By symmetry, it is sufficient to prove the assertion for the closed curve
$T_2(c)(\sigma_c^u)$, i.e. this curve is quadratically tangent to $u=0$ at
exactly one point as $|c|$ small enough. The circle $\sigma_c^u$ has the
representation in polar coordinates $\bar{x}_2 = \sqrt{2\eta_c}\cos\theta,$
$\bar{y}_2 = \sqrt{2\eta_c}\sin\theta,$ $r(c)=\sqrt{2\eta_c}\sim |c|.$ Thus,
its $T_2(c)$-image is (\ref{gloc})
$$
u_1 = a_1(c)+r(c)[\gamma(c)\cos\theta + \alpha(c)\sin\theta + O(r)], \;
v_1-v_+ = b_1(c) - r(c)[\delta(c)\cos\theta + \beta(c)\sin\theta + O(r)].
$$
First we find the points where a tangent to this curve is collinear with
the vector $(0,1)$, i.e. $u_1'(\theta)=0$. This gives the equation
$-\gamma\sin\theta + \alpha\cos\theta + O(r)=0$. It has two roots defined
up to $O(r)$
as $\theta_1 = \rho,$ $\theta_2 = \rho + \pi,$ where $\sin\rho = \alpha/\sqrt{\alpha^2+\gamma^2},$
$\cos\rho = \gamma/\sqrt{\alpha^2+\gamma^2}.$ Equating $u_1(\theta_i)=0$ we
come to the relations relative $r$: $r(c)=\pm
a_1(c)/\sqrt{\alpha^2+\gamma^2}+O(c^2)$,
where the sign is determined by that $\theta_i$ for which $r(c)>0.$ Due to assumption
$a_1'(0)\ne 0,$ we get a unique root providing the tangency of an even order.

To prove the tangency be quadratic, one needs to check that for $c$ small
enough the derivative $u_1^{\prime\prime}(v_1)\ne 0$ at the tangency
point. This derivative in the parametric form is given as (we omit subscript 1 in this calculation)
$$
u^{\prime\prime}(v)=\frac{u_\theta^{\prime\prime}v_\theta'-
u_\theta^{\prime}v_\theta^{\prime\prime}}{v_\theta'^3}=
\displaystyle{\sqrt{\alpha^2+\gamma^2}\frac{[-(\alpha\delta-\beta\gamma)+O(r)]}{\pm r(1+O(r))}}=
\pm\sqrt{\alpha^2+\gamma^2}r(c)^{-1}(1+O(r(c))).
$$
Since we saw that $r(c)\sim |c|$ as $|c|\to 0$, this derivative is as
larger in modulus as smaller $|c|$ is. Thus, we conclude the tangency be quadratic.
\end{proof}

So, we have on $D_2(c)$ two analytic curves: a spiral and a curve, both they touch the circle
$\eta=\eta_c$ at only points (generally speaking, different ones). Now let us follow a mutual
position of these two point on the circle as $c\to -0.$ The point on the curve tends
to the point $(0,0)$ as $c\to -0$ with a definite tangent. But the tangency point of the
spiral, as we shall prove below, rotates monotonically as $c\to -0$ performing infinitely
many full revolutions in the angle. This implies that the point of tangency for the spiral infinitely
many times $c_n$ pass through the point of tangency for the curve giving
quadratic tangency of he spiral and the curve (see, Fig.~\ref{fig:9}).

Let us call the unique point of tangency of the circle and the spiral on $D_2(c)$ a nose of the spiral.
Near this point, due to a quadratic tangency, the spiral is located out of the disk
bounded by the circle. Let us show that the nose of the spiral moves
monotonically in $\theta$ as $c\to -0.$

The coordinates of the nose correspond to that point of the segment $v=0$
where the $T_1(c)$-pre-image of the circle $\eta = \eta_c$ in $D_1(c)$ touches the segment.
The angle $\theta(c)$ corresponding to the nose of the spiral is calculated using the
formula $\theta(c) = \varphi(c) + \Delta_c(\eta_c)$ where the values
$(\varphi(c),\eta(c))$ have to be inserted. As we saw when proving the lemma \ref{tn}, the angle
$\varphi(c)$ has a definite limit as $c\to -0$, since the point of tangency of the curvilinear
segment and the circle $\eta = \eta_c$ on $D_1(c)$ and
the point of tangency of the circle $\eta = \eta_c$ and a curvilinear
segment on $D_2(c)$ are connected by the symmetry relation $L:(x_2,y_2)\to
(\bar{x}_2,\bar{y}_2)$ $\bar{x}_2 = -x_2, \bar{y}_2=y_2.$
We shall show that the value $\Delta_c(\eta_c)$ tends monotonically to infinity as $c\to -0$. If
so, this gives necessary conclusion on the infinite number of full
revolutions in the angle $\theta$.

The value $\eta_c$ at the tangency point is equal to
\begin{equation}\label{et}
\eta_c = \frac12(x_2^2(c)+y_2^2(c))= r^2(c)/2 =\frac{a_1^2(c)}{2(\alpha^2(c)+\gamma^2(c))}+O(c^3).
\end{equation}
Therefore, we have
$$
\Delta_c(\eta_c) = \displaystyle{(\omega + O(c^2))\ln\frac{d^2}
{-c +\omega\frac{1}{2(\alpha^2 + \gamma^2)}a_1^2(c)+O(c^3)}} \sim -\ln(-c).
$$
Thus, $\theta(c)$ depends on $c$ monotonically and increases unboundedly as $c\to -0.$
Therefore the point in $D_2(c)$, being the nose of the spiral, infinitely
many times $c_n$ coincides with the point on the same circle where the
pre-image of the segment $u=0$ from $\Pi^s(c)$ touches the circle.
\begin{figure}[h]
	\centering
	\includegraphics[width=0.5\linewidth]{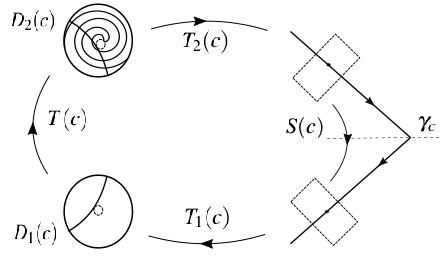}
	\caption{Case 1: Poincar\'e map as $c<0$, a mechanism of homoclinic tangency.}
\label{fig:9}
\end{figure}

The second assertion of the theorem follows from the theorem on the
existence of elliptic points near a homoclinic tangency for a symplectic map
(see, for instance, \cite{Gon,DGG}).
\begin{theorem}\label{casc}
Let $f$ be a smooth (at least $C^4$) symplectic map having a saddle fixed point $p$ and a
homoclinic orbit $f^n(q)$ through the point $q \ne p$. Suppose stable and
unstable curves of $p$ are quadratically tangent at $q$. Then for any
generic smooth one-parametric family of smooth symplectic maps
$f_\mu$ that coincides as $\mu=0$ with $f$ on any segment $[-\mu_0,\mu_0]$ there
is an integer $k_0\in \Z$ and infinitely many open intervals $I_k$, $k\ge k_0,$
such that $I_k\to 0$, as $k\to\infty$, and the map $f_\mu ,\quad \mu\in I_k$ has
a one-round elliptic homoclinic orbit (of the period $q+k$).  	
\end{theorem}
\end{proof}

\section{Hyperbolicity and ellipticity as $c>0$}

As we know, any level $V_c$ for $c>0$ small enough contains a Lyapunov saddle periodic orbit $l_c$
(remind that we assume $\omega > 0$). Its local stable $W^s(l_c)$ and unstable
$W^u(l_c)$ manifolds belong to $V_c$ and their extension by the flow occur
near $W^s(p)$ and $W^u(p),$ respectively. Therefore, they intersect cross-sections
$D_1(c)$ and $D_2(c)$ along the circles $\sigma_s(c)$, $\sigma_u(c)$.

As was indicated above, all flow orbits cutting $D_1(c)$ inside of the circle $\sigma_s(c)$
go out from $U$ not intersecting $D_2(c)$. That is why, we do not track for these
orbits. But flow orbits cutting $D_1(c)$ outside of $\sigma_s(c)$, as time increases,
do intersect $D_2(c)$ and further the cross-section $\Sigma_c.$ We construct here
a hyperbolic set that is formed near a heteroclinic connection that involve a pair of
saddle periodic orbits $l_c$, $\gamma_c$ and four transverse heteroclinic orbits, of which two
go, as time increases, from $l_c$ to $\gamma_c$ (near $\Gamma_2$), and two others -- from
$\gamma_c$ to $l_c$ (near $\Gamma_1$). Besides, there is a countable set of
transverse homoclinic orbits for every periodic orbits $l_c$ and $\gamma_c$. All this the base for
constructing the hyperbolic set.

\begin{theorem}
For $c>0$ small enough $W^s(l_c)$ and $W^u(\gamma_c)$ intersect transversely
each other along two heteroclinic orbits $\Gamma_{11}(c)$, $\Gamma_{12}(c)$ and,
by symmetry, $W^u(l_c)$ and $W^s(\gamma_c)$ intersect transversely each other along
two heteroclinic orbits $\Gamma_{21}(c)=L(\Gamma_{11}(c))$, $\Gamma_{22}(c)= L(\Gamma_{12}(c))$,
forming thereby a transverse heteroclinic connection.

The trace of $W^u(\gamma_c)$ in $\Pi^s(c)$ $($the image of the segment $v=0$ under the action of
the map $T_2(c)\circ T(c)\circ T_1(c)$$)$ consists of a pair of
spiral-shape analytic curves which wind both on the closed curve
$T_2(c)(\sigma_u(c))$ and intersect transversely the segment $u=0$ at a countable set of points
being traces of transverse Poincar\'e homoclinic orbits of the periodic
orbit $\gamma_c$.

By Smale's $\lambda$-lemma, for $n$ large enough, the closed curve
$T_2(c)(\sigma_u(c))$ contains two its segments with end points on the segment $u=0$,
whose $n$-iterations under map $S(c)$ give a countable family of analytic curves
smoothly accumulating, as $n\to \infty$, to the segment $v=0$. There is an integer $n_0(s)$
such that for $n > n_0(s)$ these curves transversely intersect the closed curve
$T_1^{-1}(c)(\sigma_s(c))$ giving a countable set of points being traces
of transverse homoclinic orbits for $l_c$ $($see, Fig.~\ref{fig:10}$)$.
\end{theorem}
\begin{proof}
Consider first $T_2(c)$-image of the circle $\sigma_u(c)$ on the disk $\Sigma_c$.
Recall that radius of the circle $\sigma_u(c)$ is of the order $\sqrt{c}$, since it defined
by the root of the equation $a_c(\eta)= -c+\omega\eta + O_2(c,\eta)=0$. Hence, we get
$\eta(c)= c/\omega + O(c^2).$ On the other hand, due to an analytic dependence of $T_2(c)$
in $c$, the $T_2(c)$-image of the center $(0,0)$ analytically depends in $c$. Therefore,
the distance from the point $(0,0)$ to the curve being
the $T_2(c)$-pre-image of the segment $u=0$ in $\Pi^s(c)$ has the order $c^l,$ $l\ge 1.$
Moreover, $l=1$ if the inequality $a'_1(0)\ne 0$ holds (see above).
This implies that the curve $T_2(c)(\sigma_u(c))$ intersects, for $c$ small enough,
segment $u=0$ transversely at two points. Indeed, the curve $T_2(c)(\sigma_u(c))$ can
be written in a parametric form with parameter $\theta$ as
$$
u = a_1(c)+r(c)[\gamma(c)\cos\theta + \alpha(c)\sin\theta + O(r)], \;
v-v_+ = b_1(c) - r(c)[\delta(c)\cos\theta + \beta(c)\sin\theta + O(r)],
$$
where $r(c)= \sqrt{2\eta(c)}=\sqrt{c/\omega + O(c^2)}.$ Equating $u_1 =0$ to find
intersection points with segment $u=0$ and dividing both sides at
$r(c)\sqrt{\alpha^2 + \gamma^2}$ we come to the equation w.r.t. $\theta$

$$A(c)+ \cos(\theta - \rho)+ O(c)=0,\; A(c)= a_1(c)/r(c)\sim \sqrt{c},$$
that has two simple roots for $c$ small enough. These simple roots correspond to
two transverse intersection points. The flow orbits through these points
are just $\Gamma_{21}(c)$, $\Gamma_{22}(c)$.
By the reversibility of the map, the closed curve $T_1^{-1}(c)(\sigma_s(c))$ intersects
transversely at two points the segment $v=0$ in $\Pi^u(c)$ as well. The flow orbits through
these intersection points are $\Gamma_{11}(c)$, $\Gamma_{12}(c)$.

Now consider the curvilinear segment being the $T_1(c)$-image of $v=0$ in
$D_1(c)$. As was proved, this curve intersect transversely at two points
the circle $\sigma_s(c)$ which divide the curve into three pieces. The
flow orbits, passing through the middle piece, leave the neighborhood of
the connection, but two remaining pieces give two analytic curves whose
$T(c)$-images are two infinite spirals on $D_2(c)$ which wind up the
circle $\sigma_u(c)$. Their $T_2(c)$-images gives two countable families
of transverse homoclinic orbits for $\gamma_c$.

In order to find transverse homoclinic orbits to $l_c$, we remark that the
segment in $\Pi^s(c)$ given as $u=\kappa>0$ for $\kappa$ small enough
intersect the closed curve $T_2(c)(\sigma_u(c))$ transversely at two
points. The same holds true for all pieces of $T_2(c)$-images of both
spirals winding up at $\sigma_u(c)$ in $D_2(c)$. Thus we have two
countable families of curvilinear segments smoothly accumulating to two
segments of the curve $T_2(c)(\sigma_u(c))$. By Smale's $\lambda$-lemma
\cite{Sm}, there is an integer $n_0 > 0$ such that all $S^n(c)$-images of
curves of both countable families intersect transversely the closed curve
$T^{-1}(c)(\sigma_s(c)$ in $\Pi^u(c)$. Thus, we have an invariant
hyperbolic set in each level $V_c$, $c>0$ small enough.
\end{proof}

\begin{figure}[h]
	\centering
	\includegraphics[width=0.6\linewidth]{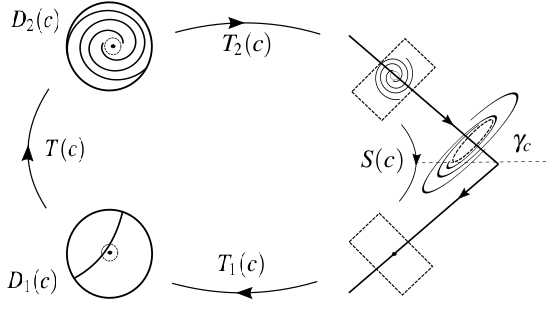}
	\caption{Poincar\'e map as $c>0$, hyperbolic set and tangency.}
\label{fig:10}
\end{figure}
The hyperbolic set, we have constructed, does not exhaust invariant sets in the
level $V_c$. One can mention the wild hyperbolic sets existing near the quadratic homoclinic
tangencies \cite{Duarte}. Here we only prove the existence of elliptic periodic orbits for
intervals of $c$-values accumulating at $c=0.$ To that end, we first prove
\begin{theorem}
There is a sequence of $c_n \to 0$ such that in the level $V_{c_n}$
the Lyapunov saddle periodic orbit $l_{c_n}$ possesses a homoclinic orbit
along which $W^s(l_{c_n})$ and $W^u(l_{c_n})$ have quadratic tangency.
\end{theorem}
\begin{proof} As the system is reversible and the related Poincar\'e map is also
reversible, it is sufficient to prove that there exist $c_n$ such that the trace
on $\Sigma_{с_n}$ of $W^u(l_{c_n})$ is a convex closed curve being quadratically
tangent to the line of $Fix(L)$ -- the diagonal $u=v$.

As was proved above, for $c_0$ sufficiently small  as $0<c\le c_0$, the trace of $W^u(l_{c})$
in $\Pi^s(c)$ is a closed curve which intersects the segment $u=0$ at two points.
The $S(c)$-pre-images of the line $u=v$ are a sequence of analytic
curvilinear segments, given as $u=u_k(v,c)$, which tend in $\Pi^s(c)$ to $u=0$
in $C^2$-topology uniformly w.r.t. $c$, as $k\to \infty$. Indeed, the
inverse iterations of $S$ are given as $u_{-n} = f^n(\zeta)u_{-n+1},$ $v_{-n}=
v_{-n+1}/f^n(\zeta),$ $\zeta = u_0v_0=u_{-1}v_{-1}=\cdots$ $=u_{-n}v_{-n}$.
The function $u_{-n}= g_n(v_{-n})$ is found as a solution w.r.t. $u$ of the equation
$u= vf^{2n}(uv)$. Multiplying both sides on $v$, we get the equation $\zeta =
v^2f^{2n}(\zeta).$ Due to form of $f(\zeta)= \nu + O_1(\zeta)$ we have an
estimate $|f|\le (1+\nu)/2 <1$ for sufficiently small $|\zeta|.$ Thus, for
any fixed $v,$ $|v-v_+|\le \delta$ we find the unique solution of the
equation $\zeta_n(v).$ This function is analytic and family tends
uniformly to zero as $n\to \infty.$ This give functions $u_n(v,c)= vf^{2n}(\zeta_n(v)).$
This family tends to zero as $[(1+\nu)/2]^{2n}.$ Since they approach to
zero in $C^2$-topology (in fact, in any $C^k$, $k\ge 2$), this implies
that for $n$ large enough the intersection of the graph of function
$u_n(v,c)$ with the closed curve $T_2(c)(\sigma_u(c)$ occurs in two points similar
as for $u=0.$

Fix some $0<c\le c_0$. For $c_0$ is small enough, the closed curve is, up
to third order terms, an ellipse in $\Pi^s(c)$ whose center approach to the
line $u=0$ with the order $c$ and its principal axes have lengths of the order $\sqrt{c}$
and their rotation angle depends as $c$ and has a limit defined by the matrix of
the linearized map $T_2(0)= T_2$. This implies this family of closed
curves intersects, as $c\to 0,$ all graphs of the functions $u_n(v,c)$ and this
intersection for the individual curve is either transversal or quadratically
tangent, or no intersection points at all. Those values of $c$ when the related
curve of the family is tangent to a fixed $u_n(v,c)$ give the values $c_n$
we search for.
\end{proof}

Now we can again to apply the Theorem \ref{casc} on the existence of elliptic
periodic point in a generic one-parameter unfolding of two-dimensional symplectic
diffeomorphisms that contains a diffeomorphism with a quadratic homoclinic tangency
\cite{Gon,DGG}).

\section{The case 2, $c>0$}

As was shown above, for the case 2 all orbits in the level $V_0$ passing through a small
neighborhood of the connection, other than those of the connection itself, leave this neighborhood
and no orbits exist which stay forever in this neighborhood. Levels $V_c,$ $c<0$, contain
no orbits at all which stay wholly in these levels, since orbits entering to the solid cylinder
through $D_1(c)$, exit from the neighborhood of the point $p$ without intersecting $D_2(c)$.
That is why we consider levels $V_c$ for $c>0$, where orbits arise lying
wholly in a neighborhood of the connection. On the corresponding transversal disk $D_1(c)$
these orbits enter to the solid cylinder through points lying inside the circle
$\sigma_s(c)$ and they exit through $D_2(c)$ inside of the circle
$\sigma_u(c)$ from another solid cylinder (see, Remark \ref{rm:case2}).

Consider the image w.r.t. the map $T_2(c)\circ T(c)\circ T_1(c)$ of the
trace of the unstable manifold $W^u(\gamma_c)$ (i.e. the segment $v=0$
from $\Pi^u(c)$). As was discussed above, generally $T_1(c)$-image of this segment
on the disk $D_1(c)$ is an analytic curvilinear segment whose distance from
the center of the disk $(0,0)$ has the order $c^l,$ $l\ge 1$, due to the analytic
dependence of the map $T_1(c)$ in $c$. If the genericity assumption above $a_1'(0)$ holds,
then $l=1$. On the disk $D_1(c)$ there is a circle $\sigma_s(c)$ defined as
$\eta = \eta(c)= c/\omega + O(c^2),$ being the trace of the stable manifold $W^s(l_c)$.
Thus, its radius is of the order $\sim \sqrt{c}.$ This implies, as above, that $\sigma_s(c)$
and the curvilinear segment (trace of $W^s(\gamma_c)$) intersect each other transversely
at two points for $c$ small enough.

Consider now that interval of the curvilinear segment which lies on $D_1(c)$ inside of the circle
$\sigma_s(c)$. Keeping in mind the modification of the formula for $\Delta_c(\eta)$
(see, Remark \ref{rm:case2}), we see that this interval (without its two extreme points
on the circle $\sigma_s(c)$) is transformed by the map $T(c)$ on $D_2(c)$
where it forms an infinite spiralling analytic curve that winds up by its both ends on the circle
$\sigma_u(c)$ (see, Fig.~\ref{fig:11}). On the same disk $D_2(c)$
there is an analytic curvilinear segment being the $T_2(c)$-pre-image of the segment
$u=0$ from $\Pi^s(c)$. The curvilinear segment intersects transversely the circle
$\sigma_u(c)$, this follows from its symmetry with the related curve in $D_1(c)$.
Since the double spiral winds up by its both ends on the circle $\sigma_s(c)$
and the segment is transverse to the circle, we get, as above, a countable set of intersection
points through which transverse homoclinic orbits of $\gamma_c$ pass.

Here we also have a countable set of intervals of $c$ on which elliptic periodic orbits exist
in $V_c$. Their proof is done by exactly the same manner as for the case 1
and $c>0$. The crucial point here is again to find a sequence of $c_n \to
0$ such that in $V_{c_n}$ a tangent symmetric homoclinic orbit of $l_c$
exists. We again iterate by the maps $S^n(c)$ on the disk $\Sigma(c)$
the closed curve $T_2(c)(\sigma_u(c))$ and find its tangency with the line
$u=v$ of the trace $Fix(L)$. The consideration is the same as in the
preceding section. Thus, we obtain
\begin{theorem}
For the case 2 there exists $c_0>0$ small enough such that on the interval $(0,c_0)$
a countable set of intervals exists whose values of $c$ correspond to
levels $V_c$ containing 1-round elliptic periodic orbit in a
four-dimensional neighborhood of the initial heteroclinic connection.
\end{theorem}
\begin{figure}[h]
	\centering
	\includegraphics[width=0.5\linewidth]{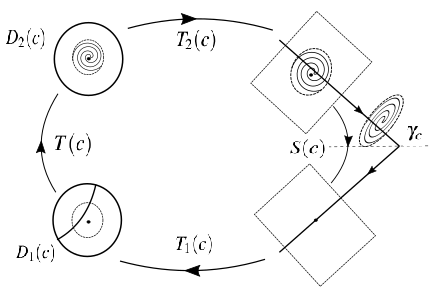}
	\caption{Case 2: Poincar\'e map as $c>0$, hyperbolic set and tangency.}
\label{fig:11}
\end{figure}

\section{1-parameter family of reversible systems: \\homoclinics of the saddle-center}

We consider in this section a generic 1-parameter family of reversible Hamiltonian systems
$X_{H_\mu}$ being an unfolding of a system that has at $\mu =0$ a heteroclinic connection
studied in Sections above. The main result here is a theorem on the
existence of a countable set of parameter values $\mu$ accumulating to
$\mu =0$ for which the related system has a homoclinic orbit of the saddle-center.
Here also it will be shown that emerging homoclinic orbits of the
saddle-center satisfy the general position conditions found in
\cite{Ler,KL1}. These conditions guarantee the existence of complicated dynamics
in the system and its non-integrability \cite{KL3}. It is worth emphasizing that
the result does not depend on what type of the connection is, the first or second one.

Recall that in the class of $C^2$-smooth Hamiltonian systems existence of
a saddle-center equilibrium is a generic phenomenon. Existence of a saddle periodic
orbit is also a generic phenomenon. But for general reversible
perturbations heteroclinic orbits joining a saddle-center and a saddle-periodic orbit
can be destroyed. We want to prove that generically here homoclinic orbits to
a perturbed saddle-center emerge. The theorem we prove can serve a criterion of
the existence of homoclinic orbits of the saddle-center. It is worth remarking
that finding such orbits is a rather delicate problem generally. Such a theorem can also be
used in the case when we deal with a two-parameter family of reversible
Hamiltonian systems and for some specific values of the parameters the
system has a connection mentioned above. Then we can find a countable set
of curves in the parameter space such that for parameters on this curve
the system has a homoclinic orbit of the saddle-center.
\begin{theorem}\label{loops}
Let $X_{H_\mu}$ be a generic one-parameter family of reversible analytic Hamiltonian
systems and at $\mu=0$ the system has a heteroclinic connection of the type 1 or 2
studied above. Then there exists a sequence of parameter values
$\mu_n$ accumulating to $\mu = 0$ such that the related system of the family has
a homoclinic orbit of the general type of a saddle-center. Related values $\mu_n$
have the same sign.
\end{theorem}
\begin{proof} Because of the reversibility, it is sufficient to prove the
existence of a sequence $\mu_n \to 0$ for which the Hamiltonian system $X_{H_{\mu_n}}$ has
an unstable separatrix of the saddle-center that intersects the cross-section
$\Sigma(\mu_n)$ at the line $Fix(L)$.

Remind that for the case of a smooth one-parameter family of reversible analytic Hamiltonian
systems being an unfolding of a system with a heteroclinic
connection, all objects under consideration, a saddle-center, a periodic
orbit in the singular level of the Hamiltonian, their stable and unstable
manifolds smoothly depend on $\mu.$ As was indicated above, Moser theorems
hold also for systems depending on parameters, therefore there is a smoothly
depending on a parameter change of variables such that in new coordinates $(x_1,y_1,x_2,y_2)$
in a neighborhood of equilibrium $p_\mu$ the Hamiltonian has the form of the analytic functions
$h_\mu$ in variables $\xi = x_1y_1$, $\eta = (x_2^2+y_2^2)/2$. The difference with
the case without parameters is in a smooth dependence on the parameter $\mu$ of coefficients
of the function $h_\mu$. Further we work in these coordinates, thus cross-sections $D_1$ и
$D_2$ to separatrices in the singular level of the Hamiltonian and $N$ can be regarded fixed
and not depending on the parameter. But in global maps zero order terms smoothly depending on
$\mu$ appear, as separatrices of the perturbed saddle-center do not lie, generally speaking,
on manifolds of the saddle periodic orbit.

The form of the perturbed global maps is as follows
$$
\begin{array}{l}
T_2(\mu):\;u_1=a(\mu)+\gamma(\mu)\bar{x}_2 + \alpha(\mu)\bar{y}_2 + \cdots,\\
\hspace{7.mm}v_1-v_+ = b(\mu)-\delta(\mu)\bar{x}_2 - \beta(\mu)\bar{y}_2 + \cdots,\\
T_1(\mu):\;x_2=b(\mu)+\alpha(\mu)(u-u_-) + \beta(\mu)v + \cdots,\\
\hspace{17.mm}y_2= a(\mu)+\gamma(\mu)(u-u_-) +\delta(\mu)v + \cdots,\\
S(\mu):\;\;u_1=u/f_\mu(\zeta), \; v_1 = vf_\mu(\zeta),\;\zeta=uv,\;f_\mu = \nu(\mu)+O_1(\zeta).
\end{array}
$$
The genericity condition for the family in coordinates means the inequality to hold
$a'(0)\ne 0.$ In virtue of the assumptions on the family we have $a(0)=0,$ $b(0)= 0,$
$\nu(0)=\nu < 1.$ Geometrically, the genericity condition means that for $\mu \ne 0$ the trace
of unstable separatrix of the saddle-center intersects the trace of the stable manifold
of the saddle periodic orbit transversely as $\mu$ varies. More spectacularly, this can be
imagined in the space $(u,v,\mu)$ where the segment $(0,0,\mu)$ represents
the 1-parameter family of saddle fixed points of the maps, the rectangular
$u=0$ represents the family of traces of stable manifolds of fixed points,
and $v=0$ is a family of traces of unstable manifolds of the saddle
fixed points. The curve of traces on the cross-section of unstable
separatrices of the saddle-centers for every small $\mu$ intersects transversely
the rectangular $u=0$ at a unique point where $\mu =0$.

Consider now the point $(0,0)$ on the disk $D_2$ being the trace of the unstable separatrix
of the perturbed saddle-center. This point under the action of the map
$S^n_(\mu)\circ T_2(\mu)$ transforms into $u_n= a(\mu)/f^n_\mu(\zeta),
\quad v_n = (v_+ +b(\mu))f^n_\mu(\zeta)$. The condition this point to lie on the
line of fixed points of involution gives the equality $u_n=v_n$, i.e.
$a(\mu)/f^n_\mu(\zeta) = (v_+ +b(\mu))f^n_\mu(\zeta)$, $\zeta = a(\mu)(v_+ +b(\mu))$.
Let us write down this equation w.r.t. $\mu$ in the form
$$
\frac{a(\mu)}{v_+ +b(\mu)}=f_\mu^{2n}(a(\mu)(v_+ +b(\mu))) > 0.
$$
Function $r(\mu)$ in the left side is defined on some neighborhood of $\mu=0$, $r(0)=0,$
$r'(0)= a'(0)/v_+ \ne 0,$ so, this function is strictly monotone. The sequence of functions
in the right side as $n\to \infty$ is a sequence of functions defined in a neighborhood of zero
and tending uniformly to zero. Thus, for large enough $n$ for those values of $\mu,$ where
$a(\mu)$ is positive, the equation for every such $n$ has a unique solution $\mu_n \to 0$
as $n\to \infty$. Asymptotics of values $\mu_n$, for which homoclinic orbits of the saddle-center
exist is as follows
$$
\mu_n=\frac{\nu^{2n}v_+}{a'(0)}.
$$

We shall now show that for these values $\mu_n$ homoclinic orbits of the saddle-center
for the related Hamiltonian system satisfy the genericity condition from \cite{Ler}.
In Moser coordinates in a neighborhood of the saddle-center  this conditions means
that the linearization matrix of the global map at the trace of a homoclinic orbit
differs from a rotation matrix.

Consider $\mu = \mu_n$. The global map in a neighborhood of the homoclinic orbit
 is the composition of maps $T_1^{\mu}\circ S^{\mu}\circ T_2^{\mu}$, i.e.
\begin{eqnarray}
x_2=[b(\mu)+\nu^{-n}(\mu)\alpha(\mu)a(\mu)-\alpha(\mu)u_- + \beta(\mu)v_+ +
\nu^n(\mu)\beta(\mu)b(\mu)+
\nonumber\\
+[\nu^{-n}(\mu)\alpha(\mu)\gamma(\mu)-\nu^n\beta(\mu)\delta(\mu)]\bar{x}_2+
[\nu^{-n}(\mu)\alpha^2(\mu)-\nu^n(\mu)\beta^2(\mu)]\bar{y}_2+ \cdots, \nonumber
\end{eqnarray}
\begin{eqnarray}
y_2= [a(\mu)+\nu^{-n}(\mu)\gamma(\mu)a(\mu)-\gamma(\mu)u_- + \delta(\mu)v_+ +\nu^n(\mu)
\delta(\mu)b(\mu)]+\nonumber\\+[\nu^{-n}\gamma^2(\mu)-\nu^n\delta^2(\mu)]\bar{x}_2+
[\nu^{-n}(\mu)\alpha(\mu)\gamma(\mu)-\nu^n(\mu)\beta(\mu)\delta(\mu)]\bar{y}_2+ \cdots \nonumber
\nonumber
\end{eqnarray}
Zero order terms of $T_1^{\mu}\circ S^{\mu}\circ T_2^{\mu}$ are equal to zero.
Denote $A=\nu^{-n}(\mu)\alpha(\mu)\gamma(\mu)-\nu^n\beta(\mu)\delta(\mu),\quad
B=\nu^{-n}(\mu)\alpha^2(\mu)-\nu^n(\mu)\beta^2(\mu), \quad C=\nu^{-n}\gamma^2(\mu)-
\nu^n\delta^2(\mu), \quad D=\nu^{-n}(\mu)\alpha(\mu)\gamma(\mu)-\nu^n(\mu)\beta(\mu)\delta(\mu)$.
Then the map $T_1^{\mu}\circ S^{\mu}\circ T_2^{\mu}$ is written as follows
$$
x_2=A\bar{x}_2+B\bar{y}_2+\dots ,\qquad y_2=C\bar{x}_2+D\bar{y}_2+\dots.
$$
Since $T_1^{\mu}, S^{\mu}, T_2^{\mu}$ are symplectic maps, then their composition
$T_1^{\mu}\circ S^{\mu}\circ T_2^{\mu}$ is also symplectic map. Let us
show that the composition $T_1^{\mu}\circ S^{\mu}\circ T_2^{\mu}$ is different from the
a rotation map. To this end, consider circles $S_1:x_2^2+y_2^2=2\eta, \quad S_2:
\bar{x}_2^2+\bar{y}_2^2=2\eta$ of the same radius in neighborhoods in $D_1,
D_2$, respectively, and prove that the circle $S_2$ under the action of
$T_1^{\mu}\circ S^{\mu}\circ T_2^{\mu}$ intersects $S_1$ at four points.

For this we use symplectic polar coordinates:
$$
x_2=\sqrt{2\eta}\cos\theta,\quad y_2=\sqrt{2\eta}\sin\theta
$$
$$\bar{x}_2=\sqrt{2\eta}\cos\varphi,\quad \bar{y}_2=\sqrt{2\eta}\sin\varphi$$
Then the condition of intersection $S_1$ and $T_1^{\mu}\circ S^{\mu}\circ T_2^{\mu}(S_2)$ is
equivalent to the equation:
\begin{eqnarray}\label{g1}
1=(A\cos\varphi+B\sin\varphi)^2+(C\cos\varphi+D\sin\varphi)^2 \Leftrightarrow\nonumber \\
\frac{A^2+B^2+C^2+D^2}{2}+\frac{A^2+C^2-B^2-D^2}{2}\cos2\varphi+(AB+CD)\sin2\varphi=1
\Leftrightarrow  \nonumber \\ \sin(2\varphi+F)=\frac{2-(A^2+B^2+C^2+D^2)}
{\sqrt{4(AB+CD)^2+(A^2+C^2-B^2-D^2)^2}}
\end{eqnarray}
The equation (\ref{g1}) has four solutions on the segment $[0,2\pi)$, if the inequality holds
$$
(2-(A^2+B^2+C^2+D^2))^2<4(AB+CD)^2+(A^2+C^2-B^2-D^2)^2.
$$
Taking into account that $A=D,\quad AD-BC=1 $ we get  $2-2A^2-B^2-^2<0$. Insert into the last
inequality
$A,B,C\qquad 2-2(\nu^{-n}\alpha(\mu)\gamma(\mu)-\nu^n\beta(\mu)\delta(\mu))^2-(\nu^{-n}
(\mu)\alpha^2(\mu)-\nu^n(\mu)\beta^2(\mu))^2-(\nu^{-n}\gamma^2(\mu)-\nu^n\delta^2(\mu))^2<0$.
After some calculations  we get $2-(\alpha^2(\mu)+\gamma^2(\mu))^2\nu^{-2n}(\mu)+
2(\alpha^2(\mu)\beta^2(\mu)-\gamma^2(\mu)\delta^2(\mu))^2-
(\beta^2(\mu)+\delta^2(\mu))^2\nu^{2n}(\mu)<0$. The last inequality will hold
beginning since some $n>n_0$ as $(\beta^2(\mu)+\delta^2(\mu))^2\nu^{2n}(\mu)\rightarrow 0,
\quad (\alpha^2(\mu)+\gamma^2(\mu))^2\nu^{-2n}(\mu)\rightarrow \infty$ as $n\rightarrow+\infty$.
Thus we have proved the map $T_1\circ S\circ T_2$ have its linear part different from
the rotation matrix.
\end{proof}

\section{Conclusion}

In the paper we study some dynamical phenomena in a one-parameter
unfolding reversible Hamiltonian systems which contain a system with a
symmetric heteroclinic connection that involves a symmetric saddle-center,
a symmetric saddle periodic orbit in the same level of the Hamiltonian and
a pair of heteroclinic orbits joining the saddle-center and periodic orbit
and permutable by the reversible involution. We found several types of
hyperbolic sets, cascades of elliptic periodic orbits and countable
sets of parameters for which homoclinic orbits to the saddle-center exist.
All this characterized the chaotic orbit behavior of the related systems.

\section{Acknowledgement}

The work by L.Lerman was partially supported by the Laboratory of Topological Methods
in Dynamics NRU HSE, of the Ministry of Science and Higher Education of RF, grant
\#075-15-2019-1931 and by project \#0729-2020-0036. Proofs of Theorems 6,8
were performed by K.N. Trifonov under a support of the Russian Science Foundation
(project 19-11-00280), Theorem 4 was proved by him under support of RFBR
(grants 18-29-10081, 19-01-00607).

\section{Data availability}

The data that support the findings of this study (proofs, pictures) are
placed in the body of the text. If some extra requirements appear, they
should addressed to the corresponding author.

\end{document}